\documentclass[aos]{imsart}

\RequirePackage{amsthm,amsmath,amsfonts,amssymb}
\RequirePackage[numbers,sort&compress]{natbib}
\RequirePackage[colorlinks,citecolor=blue,urlcolor=blue]{hyperref}
\RequirePackage{graphicx}

\startlocaldefs
\theoremstyle{plain}

\newtheorem{theorem}{Theorem}[section]

\theoremstyle{remark}
\newtheorem{definition}[theorem]{Definition}

\renewcommand{\cite}{\citet*}

\newcommand{\bX}{\mathbf{X}}

\newcommand{\bZ}{\mathbf{Z}}
\newcommand{\bD}{\mathbf{D}}

\newcommand{\bbD}{\mathbb{D}}
\newcommand{\bbP}{\mathbb{P}}
\newcommand{\bM}{\mathbf{M}}

\newcommand{\bG}{\mathbf{G}}

\newcommand{\be}{\mathbf{e}}
\newcommand{\br}{\mathbf{r}}

\newcommand{\bSi}{\boldsymbol{\Sigma}}

\newcommand{\bA}{{\bf A}}
\newcommand{\bB}{{\bf B}}
\newcommand{\bC}{{\bf C}}
\newcommand{\bS}{{\bf S}}

\newcommand{\bT}{{\bf T}}

\newcommand{\bx}{{\bf x}}
\newcommand{\by}{{\bf y}}

\newcommand{\bz}{{\bf z}}
\newcommand{\bd}{{\bf d}}
\newcommand{\bp}{{\bf p}}
\newcommand{\balpha}{{\boldsymbol\alpha}}
\newcommand{\bbeta}{{\boldsymbol\beta}}

\newcommand{\mb}{\mathbf}

\newcommand{\bss}{{\mb \Theta_n^{1/2}}}

\newcommand{\rtr}{{\rm tr}}
\newcommand{\rre}{{\rm E}}

\newcommand{\bI}{{\bf I}}

\renewcommand{\(}{\left(}
\renewcommand{\)}{\right)}

\newcommand{\rCov}{{\rm Cov}}

\DeclareMathOperator{\tr}{tr}         
      
        \DeclareMathOperator{\E}{E}

\numberwithin{equation}{section}  

\newtheorem{thm}{Theorem}[section]
\newtheorem{lem}{Lemma}[section]
\newtheorem{rem}{Remark}[section]
\newtheorem{cor}{Corollary}[section]

\allowdisplaybreaks      

\newcommand{\ycomment}{}

\endlocaldefs

\begin{document}

\begin{frontmatter}
\title{Spectral Analysis of Gram Matrices with Missing at Random Observations: Convergence, Central Limit Theorems, and Applications in Statistical Inference}
		\runtitle{Covariance matrix under missing observations}

\begin{aug}
\author[A]{\fnms{Huiqin} \snm{ Li}
			\ead[label=e1]{lihq118@nenu.edu.cn}},
\author[B]{\fnms{Guangming} \snm{ Pan}
			\ead[label=e2]{gmpan@ntu.edu.sg}}{\ycomment ,}
\author[A]{\fnms{Yanqing} \snm{ Yin}
			\ead[label=e3]{yinyq@cqu.edu.cn}}			
			\and
\author[C]{\fnms{Wang} \snm{ Zhou}
			\ead[label=e4]{wangzhou@nus.edu.sg}}
			
\runauthor{Li, Pan, Yin and Zhou.}
\address[A]{School of Mathematics and Statistics,
			Chongqing University\printead[presep={,\ }]{e1,e3}}

\address[B]{Division of Mathematical Sciences, Nanyang Technological University \printead[presep={,\ }]{e2}}

\address[C]{Department of Statistics and Data Science,  National University of Singapore \printead[presep={,\ }]{e4}}
\end{aug}

\begin{abstract}
Motivated by the statistical inference using the Gram matrix in the context of missing at random observations, this paper investigates the spectral properties of the random matrices $\mb S_n=\frac{1}{n}\mb Z\mb Z^*$, where $\mb Z=\mb D\circ(\boldsymbol{\Sigma^{1/2}}\mb X)$ represents a Hadamard random matrix with entries determined by independent Bernoulli variables $\mb D$. Operating within the high-dimensional framework, we establish the convergence of the empirical spectral distribution of $\mb S_n$ to a well-defined limiting distribution. In addition, we explore the impact of the missing mechanism on the second-order properties of the spectral distribution of the Gram matrix $\mathbf{S}_n$. We establish the central limit theorem for the linear spectral statistics of $\mathbf{S}_n$, shedding light on their fluctuations. Surprisingly, our analysis reveals that even in the ideal Gaussian distribution scenario, the fluctuations of statistics generated by eigenvalues are influenced by the eigenvectors of the population covariance matrix in the missing-at-random case. This discovery uncovers a remarkable phenomenon that starkly contrasts with the classical case. Subsequently, we demonstrate the practical application of our central limit theorem in hypothesis testing for the population covariance matrix.  \end{abstract}



\end{frontmatter}
\section{Introduction and motivation}
	
The analysis of high-dimensional data has gained significant traction across various fields, including genomics, finance, image processing, and social networks. See, for instance, \cite{ChenQ10T,CaiL14T,Wang2015high,Fan2020}. However, real-world data collection processes are often imperfect, resulting in missing observations that can profoundly affect the reliability and validity of statistical analysis.

Missing data poses a pervasive problem in numerous domains. For example, in healthcare, patient data frequently contains missing values due to incomplete medical records or patients dropping out of studies. Financial datasets may exhibit missing data due to irregular trading patterns or incomplete financial records. Similarly, in the social sciences, non-response in surveys or incomplete data collection processes can lead to missing data. These instances highlight the practical relevance of addressing the challenges posed by missing data and emphasize the necessity for robust methodologies to handle such scenarios. A pragmatic approach to addressing missing data is the exclusion of variables with incomplete observations from the analysis, thereby restricting the analysis to a subset of fully observed variables. However, in the context of gene expression data, where a substantial proportion, approximately 90\%, of genes exhibit missing values, this strategy would yield an insufficient number of variables to conduct a valid statistical analysis. Disregarding variables with only a few missing observations would also lead to a squandering of valuable information that could otherwise be utilized effectively. Considering these challenges, alternative methodologies need to be explored to handle missing data in the fields such as gene expression analysis effectively, ensuring the utilization of available information while maintaining statistical integrity. Significant progress has been made in this field, with numerous efforts dedicated to advancing the research in this direction. Noteworthy contributions can be found in the works of \cite{lounici2014high,chen2021bridging,tony2019high} and references therein.

In the realm of high-dimensional data analysis, Random matrix theory (RMT) assumes a pivotal role by offering indispensable tools and theoretical frameworks to comprehend intricate matrices encountered across diverse applications. RMT has demonstrated its applicability in a broad spectrum of fields, including wireless communications, signal processing, finance, genomics, and machine learning.  By exploring the spectral properties of random matrices that arise in various high-dimensional statistical scenarios, RMT equips researchers with powerful statistical tools to make inference about the underlying data structure and extract meaningful information. This becomes particularly valuable when dealing with large dimensions and when the matrix entries exhibit specific randomness or correlation structures. For further exploration of this research direction, we direct readers to the works of \cite{Bai2004a,Hu2019}, and other relevant sources.

In this study, our primary objective is to establish a profound connection between random matrix theory (RMT) and the analysis of incomplete data. We achieve this by investigating the spectral properties of high-dimensional Gram matrices in the presence of missing observations. Our goal is to derive theoretical results that thoroughly characterize the behavior of the spectrum of these matrices under missing-at-random mechanisms. By harnessing the tools and insights offered by RMT, we can gain a deeper understanding of how missing probabilities impact both the first and second-order limiting properties of the Gram matrix.

Our work represents a significant expansion upon the existing literature, effectively addressing a crucial research gap in this field. While previous studies, exemplified by the work of \cite{JurczakR17S}, have made valuable contributions by establishing first-order results for the spectrum of a covariance-type matrix under missing observations, our research takes a notable stride forward. We delve into the application of random matrix theory in the context of high-dimensional statistics, seeking to unlock its full potential and provide a comprehensive understanding of the effects of missing data. This advancement is pivotal in enhancing the accuracy and reliability of statistical inference, particularly in complex and high-dimensional settings.
The contributions of this paper can be summarized as follows:
\begin{enumerate}
	\item We offer a comprehensive examination of the spectral properties of high-dimensional Gram matrices when confronted with missing observations, thereby extending the current knowledge derived from random matrix theory. Due to the absence of a prior bound on the spectral norm, we employ a combination of $\epsilon$-net arguments and the analytic method of Stieltjes transform to achieve the separation of the spectrum.
	\item We derive explicit expressions for the eigenvalue distribution and the central limit theorem (CLT) for the limiting spectral statistics (LSS) under missing at random mechanisms. These expressions provide a deeper understanding of how missingness influences the spectral behavior of the Gram matrices.
	\item We discuss the test procedure designed specifically for testing the equality of the population covariance matrix with a given matrix. This procedure offers an approach to assessing covariance structures in the presence of missing data. 
\end{enumerate}

In this paper, we adopt the following notation conventions:
\begin{itemize}
	\item The spectral norm of a matrix $\mathbf{A}$ is denoted as $|\mathbf{A}|_s$ and the {\ycomment determinant} of a matrix $\mathbf{A}$ is denoted as $\det\mathbf{A}$.
	\item 
The Euclidean norm of a vector $\boldsymbol{\alpha}$ is represented as $|\boldsymbol{\alpha}|$. 
\item 
The symbols $c$ and $C$ denote absolute constants, which may vary in different instances.
\item 
The notation $\mathbf{e}_k$ refers to the $k$-th column of an identity matrix.
\item 
The transpose of a vector $\boldsymbol{\alpha}$ is denoted as $\boldsymbol{\alpha}^T$.
\item If for any $\ell>0$, $1-{\rm P}(A_n)=o(n^{-\ell})$ as $n\to\infty$, then we say $A_n$ occurs with high probability (w.h.p). 
\end{itemize}
Please note that these notation conventions have been employed consistently throughout the paper.

The paper is structured as follows. Section \ref{mr} presents theoretical results on the spectral properties of the Gram matrix under missing at random. In Section \ref{appli}, we examine the test for covariance structure in the presence of missing data. Technical proofs are deferred to Section \ref{tecproof}, with supplementary details provided in the supplementary material. Section \ref{lemmas} contains a list of necessary lemmas to be utilized.

\section{Main theoretical results}\label{mr}
This section presents our main theorem. Firstly, we introduce the Gram random matrix model that we investigate for handling missing at random data. Additionally, we provide background information on random matrix theory. Subsequently, we present definitions, model assumptions, and lemmas that play critical roles in establishing our theorem. Finally, we present our main theorems individually.
	\subsection{Model description}
We begin by introducing some primary knowledge in random matrix theory.  
	Suppose ${\bf A_n}$ is an $n\times n$ Hermitian matrix. Then the {\bf empirical spectral distribution} (ESD) of ${\bf A_n}$ is defined by
	\begin{align*}
		F^{{\bf A_n}}(x)=\frac1n\sum_{j=1}^nI(\lambda_j\le x).
	\end{align*}
	In the asymptotic scenario where $n$ tends to infinity, if the limit of $F^{\bA_n}(x)$ exists, it is denoted as the {\textbf{Limit Spectral Distribution}} (LSD). This represents the distribution that emerges as a result of this convergence in the context of spectral analysis.
	The {\bf Stieltjes transform} of $F^{{\bf A}}(x)$ is given by
	$$m_{F^{\bA}}(z)=\int\frac1{x-z}dF^{{\bf A}}(x)=\frac1n\rtr({\bf A}-z\bI_n)^{-1}, \quad z\in\mathbb{C}^+.$$
Consider the random vector $\by=\bSi^{1/2}_n\bx$, where $\bx$ represents a $p$-dimensional random vector whose entries are independent and identically distributed (i.i.d.) random variables with zero mean and unit variance. In this context, $\by$ typically serves as a model for a $p$-dimensional population with covariance matrix $\mb\Sigma_n$.
The well-known M-P law, originally presented by  \cite{marchenko1967distribution} and extensively studied in subsequent works such as \cite{Wachter78S} and  \cite{Silverstein95S}, establishes that when the ratio $p/n$ tends to $\rho\in(0,\infty)$, the distribution $F^{\mb \Sigma_n}$ converges to {\ycomment $\mathfrak{H}$}, and the population covariance matrix is bounded in the spectral norm, the ESD of the sample covariance matrix converges weakly to a non-random probability density function. 

	To incorporate missing at random mechanisms, we introduce the random vector $\bd=(d_1,\cdots,d_p)^T$, where each $d_j$ is an independent Bernoulli random variable $B(1,\mathfrak{p}_j)$, for $j=1,\cdots,p$. Consequently, we consider the Hadamard product of the random vector $\bz=\bd\circ\by$ as the population with missing probability. This population is referred to as the \emph{incomplete population}, distinguishing it from the \emph{complete population} $\by$. To conduct statistical inference, we draw a sample of $n$ observations from the incomplete population, denoted as $\bz_1,\bz_2,\cdots,\bz_n$. Accordingly, letting $\bd_j=(d_{1j},\cdots,d_{pj})^T$ and $\bbD_j={\rm diag}\(d_{1j},\cdots,d_{pj}\)$, we compute the corresponding Gram matrix as
$$\bS_n=\frac 1n\sum_{j=1}^n\bz_j\bz_j^T=\frac1n\bZ\bZ^T=\frac 1n\sum_{j=1}^n\bbD_j\by_j\by_j^T\bbD_j=\frac1n\left(\bD_n\circ\bSi_n^{1/2}\bX_n\right)\left(\bD_n\circ\bSi_n^{1/2}\bX_n\right)^T,$$
where $\bX_n=(\bx_1,\cdots,\bx_n)=(x_{jk})$ and $\bD_n=(d_{jk})$.
{\ycomment In the context of statistical applications, we introduce an auxiliary matrix $\mb \Theta_n$. This auxiliary matrix is arbitrary, as long as it satisfies the specific assumptions outlined in our main results. The form of this matrix may vary case by case in statistical applications, depending on the particular problem under consideration. An example illustrating this variability is provided in Corollary \ref{th3}, where the corresponding result is applied to covariance matrix testing, offering insight into this aspect.} Our main theorem addresses the LSD and CLT for LSS, with a particular emphasis on the product matrix $\bS_n\mb\Theta_n$. Note that one shall obtain the corresponding results of $\bS_n$ by setting $\mb\Theta_n=\bI$.
	
	\subsection{Definitions and assumptions}
Prior to presenting our main theorems, we will provide definitions and assumptions. Define $\bbP=\rre\bbD_1={\rm diag}\(\mathfrak{p}_1,\cdots,\mathfrak{p}_p\)$ and
	\begin{align*}
		\bbP^{(2)}=&\ \rre\(\bbD_j-\bbP\)^2=\bbP-\bbP^2, \quad
		\bbP^{(3)}=\ \rre\(\bbD_j-\bbP\)^3=\bbP-3\bbP^2+2\bbP^3,\\
		\bbP^{(4)}=&\ \rre\(\bbD_j-\bbP\)^4=\bbP-4\bbP^2+6\bbP^3-3\bbP^4, \quad \bbP^{(5)}=\(\bbP^{(4)}-3\(\bbP^{(2)}\)^2\).
	\end{align*}
{\ycomment Recall that a random variable $X$ is considered sub-Gaussian if there exists a constant $C$ such that: 
\begin{align*}
\E \exp \left(X^2 / C^2\right) \leq 2.
\end{align*}  We define its sub-Gaussian norm as:\begin{align*}
\|X\|_{g}=\inf \left\{c>0: \E\exp \left({X^2}/{c^2}\right) \leq 2\right\}.
\end{align*} }
We now introduce the following assumptions.
\begin{description}
\item [Assumption A:] The random variables $x_{jk},j,k=1,2,\cdots$ are independent and identically sub-Gaussian with zero mean, unit variance and sub-Gaussian norm $\|x\|_g$. {\ycomment Let $\E |x_{11}|^4=\nu_4.$}
\item  [Assumption B:] {\ycomment The matrices $\bX_n$ and $\bD_n$ are independent, and $\mb \Theta_n$ is a non-random matrix. }
\item  [Assumption C:] The convergence regime is $y_n=p/n\rightarrow y\in (0, +\infty)$.
\item  [Assumption D:] The spectral distribution $H_n$ of the matrix $\mb \Psi_n=\bss\bT_n\bss$ where $\bT_n=\bbP\bSi_n\bbP+\bbP^{(2)}\circ\bSi_n$ weakly converges to a probability distribution $H$, as $p\to\infty$, referred as the incomplete  population spectral distribution ({\sl IPSD}). Let the spectral norm of the sequence $(\bSi_n\mb\Theta_n)$ be uniformly bounded.
\end{description}
{\ycomment
\begin{rem}
	We would like to discuss two points in order. Firstly, the assumption of sub-Gaussianity is connected to Lemma 2.3, which plays a crucial role in our analysis. In the context of missing at random scenarios, establishing a high-probability bound for the spectral norm of the Gram matrix $\mb S_n$ proves to be challenging under general moment conditions. This difficulty arises due to the limitations of applying the standard moment method. The sub-Gaussianity assumption is instrumental in overcoming this challenge by ensuring that the entries possess controlled tails. This controlled tail behavior is crucial for the application of concentration inequalities. With the sub-Gaussianity assumption, we can establish Lemma 2.3 using an $\varepsilon$-net argument. It's worth noting that the use of sub-Gaussianity is a common practice in high-dimensional statistics, serving to control tail probabilities and ensure the concentration of random matrices. Furthermore, we would like to emphasize that while the existence of a high-probability bound for $\mb S_n$ under general distributions is not guaranteed, our CLT results would seamlessly align with any advancements in this direction, provided such bounds can be established. 
	
	Secondly, the specific matrix $\mb T_n$ introduced in Assumption D takes on the role of the population covariance matrix for incomplete data. This assignment is elucidated through the following calculations:	\begin{align*}
		&\E \bbD_1\by_1\by_1^T\bbD_1\\
		=&\E \(\(\bbD_1-\bbP\)\by_1\by_1^T\(\bbD_1-\bbP\)+\(\bbD_1-\bbP\)\by_1\by_1^T\bbP+\bbP\by_1\by_1^T\(\bbD_1-\bbP\)+\bbP\by_1\by_1^T\bbP\)\\
		=&\bbP^{(2)}\circ\bSi_n+\bbP\bSi_n\bbP.
	\end{align*}
	Under the scenario of complete data where $\mathfrak{p}_j=1, 1\leq j\leq p$, we have $\bbP=\bI_p$ and $\bbP^{(2)}=\mb 0$, resulting in the degeneration of $\mb T_n$ to $\bSi_n$.
		\end{rem} }

\subsection{Lemmas on quadratic forms}
This subsection is dedicated to presenting lemmas concerning the quadratic forms of the incomplete population $\bz$. These lemmas provide valuable insights into the effects of missing data thus play a central role in analyzing the spectral properties of the missing data model. For the $p$-dimensional random vector $\bz \in \mathbb{R}^{p}$, let us recall the regular complex matrix field $\mathbb {C}^{p\times p}$, which consists of all $p$-dimensional nonrandom matrices. We define $$\mathcal{Q}(\bz,\mathcal{M}) = \bz^T{\mathcal {M}}\bz, \quad \mathcal {M}\in \mathbb {C}^{p\times p}.$$
Within this context, it is our primary objective to present the mean and variance-covariance profile of $\mathcal{Q}(\bz,\cdot)$. These profiles play a fundamental role in various aspects of our analysis. Firstly, they are instrumental in establishing the convergence of the ESD. Secondly, by leveraging the variance-covariance profiles, we can ascertain the limiting parameters that govern the second-order fluctuations of the Gram matrix. 
	\begin{lem}[First two moments of $\mathcal{Q}(\bz,\cdot)$]\label{lec2}
	Assuming $\bz$ follows our model, we consider any $p\times p$ nonrandom matrix $\bA$ and a nonrandom symmetric matrix $\bB$. In this context, we establish the following result:	
	 $\rre\mathcal{Q}(\bz,\bA)=\tr\bA\bT_n$ and
$$\rre\left[\(\mathcal{Q}(\bz,\bA)-\tr\bA\bT_n\)\(\mathcal{Q}(\bz,\bB)-\tr\bB\bT_n\)\right]=\mathcal{I}(\mb A,\mb B,\mb\Sigma_n,\mathbb{P})+(\nu_4-3)\mathcal{II}(\mb A,\mb B,\mb\Sigma_n,\mathbb{P}),$$ where the terms that are independent of the kurtosis of the underlying distribution are given by $\mathcal{I}(\mb A,\mb B,\mb\Sigma_n,\mathbb{P})=\mathcal{I}_1(\mb A,\mb B,\mb\Sigma_n,\mathbb{P})+\mathcal{I}_2(\mb A,\mb B,\mb\Sigma_n,\mathbb{P})$ with
		$\mathcal{I}_1(\mb A,\mb B,\mb\Sigma_n,\mathbb{P})=2\rtr\(\bT_n\bA\bT_n \bB \)$ and
	\begin{align*}
	&\mathcal{I}_2(\mb A,\mb B,\mb\Sigma_n,\mathbb{P})\\\notag
	=&4\rtr\left[\bbP^{(2)}\(\bA\circ\bB\)\bbP^{(2)}\(\bSi_n\circ\bSi_n\)\right]+2\rtr\left[\(\bA\circ\bbP^{(2)}\)\bSi_n\(\bB\circ\bbP^{(2)}\)\bSi_n\right]\\
	&+6\rtr\(\bbP^{(3)}\circ\bA\circ\bSi_n\circ\(\bB\bbP\bSi_n\)  \)+2\rtr\(\(\bbP^{(2)}\circ\bA\)\bSi_n\bbP \bB\bbP\bSi_n \)\\
	&+2\rtr\(\(\bbP^{(2)}\circ\bB\)\bSi_n\bbP \bA\bbP\bSi_n\)+3\rtr\left[\bbP^{(3)}\circ\bB\circ\bSi_n\circ\(\bA\bbP\bSi_n+\bA^T\bbP\bSi_n\)\right]\\
   &+3\rtr\left[\bA\circ\bSi_n\circ\bB\circ\bSi_n\circ\bbP^{(5)}\right]+4\rtr\left[\bbP^{(2)}\circ\(\bA\bbP\bSi_n+\bA^T\bbP\bSi_n\)\circ\(\bB\bbP\bSi_n\)\right],	
\end{align*}	
while the terms that are dependent on the kurtosis of the underlying distribution are expressed as
	\begin{align*}
	&\mathcal{II}(\mb A,\mb B,\mb\Sigma_n,\mathbb{P})\\\notag
	=&\rtr\left[\(\bSi_n^{1/2}\(\bbP \bA\bbP+\bbP^{(2)}\circ \bA\)\bSi_n^{1/2}\)\circ\(\bSi_n^{1/2}\(\bbP \bB\bbP+\bbP^{(2)}\circ \bB\)\bSi_n^{1/2}\)\right]\\
	&+2\rtr\left[\bbP^{(2)}\(\bA\circ\bB\)\bbP^{(2)}\(\bSi_n^{1/2}\circ\bSi_n^{1/2}\)^2\right]+\rtr\left[\bbP^{(5)}\circ\bA\circ\bB\circ\(\bSi_n^{1/2}\circ\bSi_n^{1/2}\)^2\right]\\
	&+2\rtr\left[\(\bSi_n^{1/2}\circ\bSi_n^{1/2}\circ\bSi_n^{1/2}\)\(\bA\circ\bbP^{(3)}\)\bB\bbP\bSi_n^{1/2}\right]\\
	&+\rtr\left[\(\bSi_n^{1/2}\circ\bSi_n^{1/2}\circ\bSi_n^{1/2}\)\(\bB\circ\bbP^{(3)}\)\(\bA+\bA^T\)\bbP\bSi_n^{1/2}\right]\\
	&+2\rtr\left[\(\bSi_n^{1/2}\circ
	\bSi_n^{1/2}\)\bbP^{(2)} \(\(\bA\bbP\bSi_n^{1/2}+\bA^T\bbP\bSi_n^{1/2}\)\circ\(\bB\bbP\bSi_n^{1/2}\)\)\right].	
\end{align*}	

	\end{lem}	
	Furthermore, we proceed to present a lemma regarding the higher-order moments of $\mathcal{Q}(\bz,\cdot)$. This lemma also plays crucial role in our analysis, serving multiple purposes. Firstly, it enables us to derive some necessary high probability results. Secondly, it facilitates the control of negligible terms in our analysis.

\begin{lem}[High order moments of $\mathcal{Q}(\bz,\cdot)$]\label{lec1}
		Assume $\bz$ follows our model. Let ${\bf A}=(a_{jk})$ be a $p\times p$ nonrandom symmetric matrix and  $\rre|x_j|^{\ell}\le \nu_{\ell}$. Then for $\ell>2$,
		\begin{align*}
			&\rre\left|\mathcal{Q}(\bz,\bA)-\rtr\({\bf A}\bT_n\)\right|^{\ell}
			\le C_{\ell}\left[\left(n\|\bA\|^2\right)^{\ell/2}+\nu_{2\ell}n\|\bA\|^{\ell}\right],
		\end{align*}
		where  $C_{\ell}$ is a constant depending on $\ell$ only.
	\end{lem}	
The proofs of the aforementioned lemmas are deferred to the supplement material. 

\subsection{Limiting spectral distribution}  \label{Sec:setup}

{\ycomment Note that for a random matrix, both its ESD and the corresponding probability measure are inherently stochastic. To ensure clarity and comprehensiveness in our statement, we introduce the following definitions.
\begin{definition}[Convergence with high probability] 
We say that a sequence of random variables ${x_n}$ converges with high probability to $x$, denoted as $x_n \xrightarrow{\text{w.h.p.}} x$, if for any $\epsilon>0$ and $\ell>0$,  $${\rm P}\(|x_n-x|>\epsilon\)=o(n^{-\ell}), \quad n\to \infty.$$
\end{definition}

\begin{definition}[Convergence with high probability of ESD] Consider a sequence of random matrices ${\mathbf{A}_n}$ with empirical spectral distributions denoted by $F^{\mathbf{A}_n}$. Adopting the usual definition of Kolmogorov distance $\|F - G\|_{\infty} = \sup_x |F(x) - G(x)|$ for two distribution functions $F$ and $G$, we say that the sequence $F^{\mathbf{A}_n}$ converges with high probability to $F$, denoted as $F^{\mathbf{A}_n} \xrightarrow{\text{w.h.p.}} F$, if  $$\|F^{\mb A_n}-F\|_{\infty}\xrightarrow{\text{w.h.p.}} 0.$$
\end{definition}  }
The forthcoming theorem establishes the asymptotic distribution of $F^{\bS_n\mb\Theta_n}(x)$, providing insights into the behavior of the Gram matrix under missing at random mechanisms.
\begin{thm}[LSD]\label{thm1}
Under Assumptions A-D, the ESD $F^{\bS_n\mb\Theta_n}(x)$ converges to the limiting spectral distribution $F$ with high probability. The Stieltjes transform $m(z)$ associated with $F$ is analytic in $\mathbb{C}^+$. For each $z\in \mathbb{C}^+$, $m=m(z)$ represents a unique solution to the equation:
\begin{align}\label{eqs}
{m}=\int\frac{1}{t(1-y-yz m)-z}dH(t),
\end{align}
where uniqueness is guaranteed within the set $\{m\in\mathbb{C}^+:-\frac{1-y}{z}+ym\in\mathbb{C}^+\}$.
\end{thm}

As observed, the limiting spectral distribution of the Gram matrix, under the presence of missing at random, still exhibits characteristics akin to the M-P law. However, a notable distinction arises in the distribution parameter $H$, which transitions from the complete population to the incomplete population due to the influence of missing at random. This phenomenon bears resemblance to the behavior observed in sample correlation matrices, see \cite{ElKaroui09C}.

{\ycomment
\begin{rem}
	We would like to discuss some points. Firstly, we would like to emphasize that our Gram matrix model differs from the one in \cite{JurczakR17S} in terms of normalization parameters, using $n$ instead of $N_{jk}=1 \vee \# \left\{t \in\{1, \ldots, n\}: d_{j t} d_{k t}=1\right\}$ for the entries of the Gram matrix $\left(\bD_n\circ\bSi_n^{1/2}\bX_n\right)\left(\bD_n\circ\bSi_n^{1/2}\bX_n\right)^T$. This difference in normalization brings about technical challenges in studying the spectral properties of the Gram matrix, especially under non-diagonal covariance matrix scenarios, as shown in \cite{JurczakR17S}. The dependence introduced by $N_{jk}$ breaks the independent structure of the columns in $\left(\bD_n\circ\bSi_n^{1/2}\bX_n\right)$, creating challenges in applying standard techniques for deriving the second-order spectral properties of the Gram matrix. In contrast, our studied Gram model maintains both the independent structure between columns of the data matrix and the conventional covariance matrix form. This characteristic not only facilitates a more straightforward investigation of the second-order limit but also enhances the overall clarity of the theoretical procedure. Secondly, it is evident that setting the auxiliary matrix $\mb \Theta_n$ to the identity yields a distinct LSD from the one in \cite{JurczakR17S}. More specifically, under the null case where $\mb \Sigma_n^{1/2}$ is the identity and assuming equal missing probability, our Gram matrix model remains non-negative definite, whereas the matrix model in \cite{JurczakR17S} could result in negative eigenvalues for a small enough missing probability. Importantly, our model ensures the applicability of theoretical results to statistical applications, as shown in Section \ref{appli}. Finally, the establishment of convergence with high probability beyond almost surely meets the requirement in the proof of Theorem \ref{thmnoout}, which, in turn, fulfills the requirement in the proof of Theorem \ref{thm2}.
\end{rem}
}

Let $F^{y,H}$ represent the distribution $F$, and let $F^{y_n,H_{n}}$ be obtained by replacing $y$ and $H$ with $y_n$ and $H_n$, respectively. We denote $m_n^0(z)$ as $m_{F^{y_n,H_{n}}}(z)$, which satisfies the equation (\ref{eqs}). In other words, we have:
\begin{align}
\underline{m}_n^0(z) = \frac{1}{y_n\int\frac{t}{t\underline{m}_n^0(z)+1}dH_n(t)-z}\label{i2}
\end{align}
and
\begin{align}\label{ll2}
z = -\frac{1}{\underline{m}_n^0} + y\int\frac{t}{t\underline{m}_n^0+1}d{\ycomment H_n}(t).
\end{align}
Here, $-z\underline{m}_n^0(z) = 1-y_n-y_nz m_n^0(z)$.
	
	\subsection{The separation of spectrum}
The subsequent theorem provides further understanding regarding the convergence of sample eigenvalues. We present a comprehensive theorem that demonstrates the absence of eigenvalues with high probability outside the limiting support. This result is also fundamental in establishing the central limit theorem for the LSS as it provides a high probability bound on the spectral norm of $\mb S_n$.

\begin{thm}[No outside eigenvalues]\label{thmnoout}
Consider the gram matrix $\mb S_n\mb\Theta_n$. Assume that the interval $[a, b]$ with $a>0$ lies outside the support of $F^{c, H}$ and $F^{c_n, H_n}$ for all large $n$. 
Then with high probability, no eigenvalue of $\mb S_n\mb\Theta_n$ appears in $[a, b]$ for all large $n$.
\end{thm}

To prove Theorem \ref{thmnoout}, we adopt the strategy employed by Bai and Silverstein \cite{BaiS98N}. Specifically, we begin by establishing a primary bound on the spectral norm of $\mb S_n$. Subsequently, we complete the proof of this theorem. The  proof of Theorem \ref{thmnoout} is postponed to Section \ref{tecproof}. We are in position to introduce a lemma that bounds the spectral norm of $\mb S_n$ by a constant-order quantity.

\begin{lem}\label{normlem}
Under the assumptions of Theorem \ref{thmnoout}, with high probability, we have
\begin{align*}
\|\mb S_n\|_s \leq C(1+\sqrt{y})^2.
\end{align*}
\end{lem}
To prove the above lemma, we employ the $\varepsilon$-net argument. This approach draws inspiration from the proof of Theorem 4.4.5 in \cite{vershynin2018high}, which considers a random vector with i.i.d. sub-Gaussian entries. However, our model deals with more general variables. To begin, we introduce some necessary definitions and lemmas, the proofs of which can be found in \cite{vershynin2018high}.

\begin{definition}
	[$\varepsilon$-net] Let $(T, d)$ be a metric space. Consider a subset $K \subset T$ and let $\varepsilon>0$. A subset $\mathcal{N} \subseteq K$ is called an {\bf $\varepsilon$-net} of $K$ if every point in $K$ is within distance $\varepsilon$ of some point of $\mathcal{N}$, i.e.
\begin{align*}
\forall x \in K, \quad \exists x_0 \in \mathcal{N}: d\left(x, x_0\right) \leq \varepsilon.
\end{align*}
\end{definition} 

\begin{definition}
	[Covering numbers] The smallest possible cardinality of an $\varepsilon$-net of $K$ is called the {\bf covering number} of $K$ and is denoted $\mathcal{N}(K, d, \varepsilon)$. Equivalently, $\mathcal{N}(K, d, \varepsilon)$ is the smallest number of closed balls with centers in $K$ and radii $\varepsilon$ whose union covers $K$.
\end{definition}
\begin{lem}
[Covering numbers of the Euclidean sphere] \label{net-n} The covering numbers of the unit Euclidean sphere {\ycomment $\mathbb{S}^{p-1}$} with Euclidean distance satisfy that for any $\varepsilon>0$ :
\begin{align*}
\left(\frac{1}{\varepsilon}\right)^p \leq \mathcal{N}\left(\mathbb{S}^{p-1}, \varepsilon\right) \leq\left(\frac{2}{\varepsilon}+1\right)^p.
\end{align*}	
\end{lem}

\begin{lem}
[Quadratic form on a net] \label{net-norm} Let $\mb A$ be a $p \times n$ matrix and $\varepsilon \in[0,1 / 2)$.
For any $\varepsilon$-net $\mathcal{P}$ of the sphere $\mathbb{S}^{p-1}$ and any $\varepsilon$-net $\mathcal{N}$ of the sphere $\mathbb{S}^{n-1}$, we have
\begin{align*}
\sup _{\balpha \in \mathcal{P}, \bbeta \in \mathcal{N}} \balpha^T\mb A\bbeta \leq\|\mb A\| \leq \frac{1}{1-2 \varepsilon} \sup _{\balpha \in \mathcal{P}, \bbeta \in \mathcal{N}}\balpha^T\mb A\bbeta .
\end{align*}	
\end{lem}

	\begin{proof}[Proof of Lemma \ref{normlem}]
		Our objective is to ensure control over $\balpha^T (\bD_n\circ\bSi_n^{1/2}\bX_n)\bbeta$ for all vectors $\balpha$ and $\bbeta$ belonging to the unit sphere, as defined by the {\ycomment operator norm}. To accomplish this, we adopt a two-step approach. First, we establish a rigorous control over $\balpha^T (\bD_n\circ\bSi_n^{1/2}\bX_n)\bbeta$ for fixed vectors $\balpha$ and $\bbeta$. This step allows us to obtain precise bounds on the expression. Second, we discretize the unit sphere by utilizing a carefully chosen net, which enables us to approximate the sphere with a finite set of representative points. By leveraging this discretization, we can then employ a union bound over all pairs of $\balpha$ and $\bbeta$ within the net. By taking into account the tight control established in the first step and the discretization scheme in the second step, we can confidently conclude the proof.

\subsubsection {Bound for Fixed vectors} Let $r_{ij}=\be_i^T\mb \Sigma_n^{1/2}\be_j$. For any fixed $\balpha \in \mathbb{S}^{p-1}$ and $\bbeta \in \mathbb{S}^{n-1}$, {\ycomment we may write the expression} $ \balpha^T (\bD_n\circ\bSi_n^{1/2}\bX_n)\bbeta$ as follows:
\begin{align*}
\balpha^T (\bD_n\circ\bSi_n^{1/2}\bX_n)\bbeta=\sum_{j=1}^n\sum_{k=1}^p \(\sum_{i=1}^p d_{ij} \alpha_i r_{ik}\)\beta_jx_{kj}\triangleq \sum_{j=1}^n\sum_{k=1}^p \varrho_{kj}\beta_jx_{kj},
\end{align*}
where $\varrho_{kj}=\sum_{i=1}^p d_{ij} \alpha_i r_{ik}$
is a weighted sum of independent  Bernoulli variables $d_{ij}$. 
To bound $\sum_{j=1}^n\sum_{k=1}^p \varrho_{kj}\beta_jx_{kj}$, we observe that the variables $x_{kj}$ are independent sub-Gaussian variables with sub-Gaussian norm $\|x\|_g$. By utilizing the concentration properties of sub-Gaussian random variables (Lemma \ref{sg-2}) and the properties of sub-Gaussian norms (Lemma \ref{sg-1}), for any $ \Upsilon \geq 0$ {\ycomment and for all choices of $d^{\odot}_{ij} \in \{0,1\}^{pn}$, we have}:
\begin{align*}
\mathrm{P}\{\balpha^T (\bD_n\circ\bSi_n^{1/2}\bX_n)\bbeta \geq \Upsilon|d_{ij}=d_{ij}^\odot,i=1,\cdots, p,j=1,\cdots,n\} \leq 2 \exp \left(-c \Upsilon^2 / C_{g}^2\right), 
\end{align*}

where \begin{align*}
	C_g^2\leq &C\sum_{j=1}^n\sum_{k=1}^p \varrho_{kj}^2\beta_j^2=C\sum_{j=1}^n\sum_{k=1}^p \sum_{i_1,i_2=1}^p d_{i_1j}^\odot d_{i_2j}^\odot \alpha_{i_1}\alpha_{i_2} r_{i_1k}r_{i_2k}\beta_j^2\\\notag
	\leq &C\balpha^T\mb\Sigma_n\balpha\leq C\|\mb \Sigma_n\|_s.
\end{align*}
Thus, we arrive at 
\begin{align*}
&\mathrm{P}\{\balpha^T (\bD_n\circ\bSi_n^{1/2}\bX_n)\bbeta \geq \Upsilon\}\\\notag=&\sum_{d_{ij}^\odot\in\{0,1\}^{pn}} \mathrm{P}\{\balpha^T (\bD_n\circ\bSi_n^{1/2}\bX_n)\bbeta \geq \Upsilon|d_{ij}=d_{ij}^\odot\}\mathrm{P}\{d_{ij}=d_{ij}^\odot\} \leq  \exp \left(-C \Upsilon^2 \right),
\end{align*}for some constant $C$.

\subsubsection{Union Bound:} Let $\varepsilon=1 / 4$. Utilizing Lemma \ref{net-n}, we can construct an $\varepsilon$-net denoted as $\mathcal{P}$ for the sphere $\mathbb{S}^{p-1}$ and an $\varepsilon$-net denoted as $\mathcal{N}$ for the sphere $\mathbb{S}^{n-1}$, where the cardinalities are bounded as follows:
\begin{align*}
|\mathcal{P}| \leq 9^p \quad \text { and } \quad|\mathcal{N}| \leq 9^n.
\end{align*}
By further employing Lemma \ref{net-norm}, we can establish an upper bound on the spectral norm of $\bD_n\circ\bSi_n^{1/2}\bX_n$ using these nets as follows:
\begin{align*}
\|\bD_n\circ\bSi_n^{1/2}\bX_n\|_s \leq 2 \max _{\balpha \in \mathcal{P}, \bbeta \in \mathcal{N}}\balpha^T\(\bD_n\circ\bSi_n^{1/2}\bX_n\)\bbeta .
\end{align*}
By utilizing the bound for fixed vectors, we obtain:
\begin{align*}
&\mathrm{P}\left\{\|\bD_n\circ\bSi_n^{1/2}\bX_n\|_s \geq 2C_1\(\sqrt{p}+\sqrt{n}\)\right\} \\\notag
\leq& \sum_{\balpha \in \mathcal{P}, \bbeta \in \mathcal{N}} \mathbb{P}\bigg\{\balpha^T\(\bD_n\circ\bSi_n^{1/2}\bX_n\)\bbeta \geq C_1\(\sqrt{p}+\sqrt{n}\)\bigg\}
\leq  9^{p+n}\exp \left(-CC_1^2 \(\sqrt{p}+\sqrt{n}\)^2 \right).
\end{align*}
If we choose the constant $C_1$ to be sufficiently large, we can conclude that:
\begin{align*}
\mathrm{P}\left\{\|\bD_n\circ\bSi_n^{1/2}\bX_n\|_s \geq 2C_1\(\sqrt{p}+\sqrt{n}\)\right\} \leq \exp \left(-n\right) .
\end{align*}
Finally, we can conclude that with high probability, the spectral norm of {\ycomment $\mb S_n$} is bounded as follows:
\begin{align*}
\|{\ycomment \mb S_n}\|_s \leq C(1+\sqrt{y})^2.
\end{align*}
Hence, the proof of this lemma is complete.

	\end{proof}

	\subsection{CLT for the LSS}
Let us consider the LSS of $\mb S_{n}\mb \Theta_n$ associated with a test function $f$. This statistic is given by the integral:

\begin{align*}
L_n(f, \mb {S}_{n}\mb \Theta_n) = \int f(x) dF^{\mb {S}_{n}\mb \Theta_n}(x).
\end{align*}
This kind of statistics arise in many statistical problems, especially in hypotheses tests considering the covariance structure of the {\ycomment population}.
In order to establish the second-order convergence of the LSS, a centralization step is performed to eliminate the leading order term. This is achieved by subtracting the integral of $f(x)$ with respect to $dF^{c_n, H_n}(x)$ from the integral of $f(x)$ with respect to $dF^{\mathbf{S}_{n}\mb \Theta_n}(x)$.
Consequently, the centralized LSS can be expressed as follows:
\begin{align*}
L^{c}(f, \mb{S}_{n}\mb \Theta_n) = \int f(x) d[F^{\mb{S}_{n}\mb \Theta_n}(x)-F^{c_n, H_n}(x)].
\end{align*}
Here, $F^{\mb{S}_{n}\mb \Theta_n}(x)-F^{c_n, H_n}(x)$ provides a measure of the deviation from the finite-dimensional approximation of the LSD. By examining the convergence properties and establishing the second-order convergence of this quantity, we can gain insights into the behavior of the LSS with respect to the underlying parameters and better understand its statistical properties. 

We now present the formal results that establish the limiting fluctuation of a collection of LSSs in detail. This convergence theorem allows us to make reliable inference about the underlying population parameters based on the observed LSSs.

	\begin{thm}[CLT for LSS]\label{thm2}
	 Let $f_1,\cdots,f_{\kappa}$ be functions on $\mathbb{R}$ analytic on an open interval containing
		$
		\left[\liminf_n\lambda_{\min}^{\bSi_n\mb \Theta_n}I_{(0,1)}(y)\left(1-\sqrt y\right)^2, 
			\limsup_n\lambda_{\max}^{\bSi_n\mb \Theta_n}\left(1+\sqrt y\right)^2\right].
$
			Under Assumptions A-D, we have
		that  the $\kappa$ dimensional random vector $\(L^{c}(f_j, \mb{S}_{n}\mb \Theta_n)\)_{j=1}^{\kappa}$
		is tight and converges weakly to a Gaussian vector $\(X_{f_j}\right)_{j=1}^{\kappa}$. {\ycomment The respective expectation is }
		\begin{align}\label{ee}
			\rre X_{f}=&-\frac1{2\pi i}\oint f(z)\left\{\frac{y\int\frac{\underline m^3(z)t^2}{\({\underline m}(z)t+1\)^3}dH(t)}{\(1-y\int\frac{\underline m^2(z)t^2}{\({\underline m}(z)t+1\)^2}dH(t)\)^2}+\frac{\underline m^3(z)a(z)}{1-y\int\frac{\underline m^2(z)t^2}{\({\underline m}(z)t+1\)^2}dH(t)}\right\}dz,
		\end{align}
		{\ycomment while the covariance is given by }
		\begin{align}\label{cc}
			\begin{split}
			{\rm Cov}\left( X_{f},X_g\right)=&-\frac1{2\pi^2}\oint\oint{f(z_1)g(z_2)}\(\frac{\frac {d\underline m(z_1)}{dz_1}\frac {d\underline m(z_2)}{dz_2}}{{\(\underline m(z_2)-\underline m(z_1)\)^2}}-\frac1{\(z_1-z_2\)^2}\) dz_1dz_2\\
			&-\frac1{4\pi^2}\oint
			\oint f(z_1)g(z_2)\frac{\partial^2\underline m(z_1)\underline m(z_2)d_2(z_1,z_2)}{\partial z_2\partial z_1}dz_1dz_2,
			\end{split}
		\end{align}
		where $f,g\in\left\{f_1,\cdots,f_{\kappa}\right\}$. The limiting parameters can be expressed as follows by denoting $\mathbb{T}_n^{(1)}(\mb\Theta_n,\bSi_n,\mathbb{P}_n,z) = \bss\left(\underline m(z)\mb\Psi_n + \mathbf{I}_p\right)^{-1}\bss$ and $\mathbb{T}_n^{(2)}(\mb\Theta_n,\bSi_n,\mathbb{P}_n,z) = \bss\left(\underline m(z)\mb\Psi_n + \mathbf{I}_p\right)^{-2}\bss$:
		\begin{align*}
		a(z)
		=&\lim\limits_{n\to\infty}{n^{-1}}\mathcal{I}_2\(\mathbb{T}_n^{(1)}(\mb\Theta_n,\bSi_n,\mathbb{P}_n,z),\mathbb{T}_n^{(2)}(\mb\Theta_n,\bSi_n,\mathbb{P}_n,z),\bSi_n,\mathbb{P}_n\)\\&+(\nu_4-3)\lim\limits_{n\to\infty}{n^{-1}}\mathcal{II}\(\mathbb{T}_n^{(1)}(\mb\Theta_n,\bSi_n,\mathbb{P}_n,z),\mathbb{T}_n^{(2)}(\mb\Theta_n,\bSi_n,\mathbb{P}_n,z),\bSi_n,\mathbb{P}_n\)
	.
	\end{align*}	
		\begin{align*}
		d_2(z_1,z_2)
		=&\lim\limits_{n\to\infty}{n^{-1}}\mathcal{I}_2\(\mathbb{T}_n^{(1)}(\mb\Theta_n,\bSi_n,\mathbb{P}_n,z_1),\mathbb{T}_n^{(1)}(\mb\Theta_n,\bSi_n,\mathbb{P}_n,z_2),\bSi_n,\mathbb{P}_n\)\\&+(\nu_4-3)\lim\limits_{n\to\infty}{n^{-1}}\mathcal{II}\(\mathbb{T}_n^{(1)}(\mb\Theta_n,\bSi_n,\mathbb{P}_n,z_1),\mathbb{T}_n^{(1)}(\mb\Theta_n,\bSi_n,\mathbb{P}_n,z_2),\bSi_n,\mathbb{P}_n\)
	.
	\end{align*}	
	 Moreover, both in equation (\ref{ee}) and (\ref{cc}), the contours are closed and taken in the positive direction in the complex plane. Each contour encloses the support of $F^{y, H}$, with the two contours in equation (\ref{cc}) assumed to be non-overlapping. 
	\end{thm}
	
\begin{rem}
A thorough comparison between our theorem and Theorem 1.1 in \cite{Bai2004a} reveals the noteworthy impact of missing at random on the CLT for LSSs. It becomes apparent that the inclusion of  additional terms in both the means and variance-covariance functions can introduce significant bias if one simply treats $\mathbf{T}_n$ as the population covariance matrix for incomplete data and applies the ordinary CLT without considering the influence of missing at random on the second-order properties under the non-null case. In contrast to the case of complete observations, where the CLT for LSS is primarily influenced by the eigenvalues of the population covariance matrix $\mathbf{\Sigma}$ under the assumption of a Gaussian distribution, the situation changes when dealing with missing observations. Even under the ideal assumption of Gaussianity, the CLT for LSS under the non-null case is strongly affected by the eigenvectors of $\mathbf{\Sigma}_n$ {\ycomment and $\Theta_n$, not just their eigenvalues.} 

{\ycomment To provide a more lucid explanation and isolate the impact of the auxiliary matrix $\mb \Theta_n$, we set $\mb \Theta_n=\mb I_p$ in the aforementioned CLT. In this case, we have:  \begin{align*}
		a(z)
		=&\lim\limits_{n\to\infty}{n^{-1}}\mathcal{I}_2\(\left(\underline m(z)\mb\bT_n + \mathbf{I}_p\right)^{-1},\left(\underline m(z)\mb\bT_n + \mathbf{I}_p\right)^{-2},\bSi_n,\mathbb{P}_n\)\\&+(\nu_4-3)\lim\limits_{n\to\infty}{n^{-1}}\mathcal{II}\(\left(\underline m(z)\mb\bT_n + \mathbf{I}_p\right)^{-1},\left(\underline m(z)\mb\bT_n + \mathbf{I}_p\right)^{-2},\bSi_n,\mathbb{P}_n\)
	.
	\end{align*}	 
Taking into account the definitions of $\mathcal{I}_2$ and $\mathcal{II}$ in Lemma \ref{lec2}, we observe that $a(z)$ is determined by numerous terms and affected not only by the eigenvalues of $\mathbf{\Sigma}_n$ but also by its eigenvectors. An illustrative example of such a term is $$\rtr\left[\bbP^{(2)}\(\left(\underline m(z)\mb\bT_n + \mathbf{I}_p\right)^{-1}\circ\left(\underline m(z)\mb\bT_n + \mathbf{I}_p\right)^{-2}\)\bbP^{(2)}\(\bSi_n\circ\bSi_n\)\right].$$ The discussion on $d_2(z_1,z_2)$ is similar.
For a more profound understanding, it is crucial to note that in the case of complete observations under Gaussian distribution, the influence on the LSSs from the left eigenmatrix of $\mathbf{\Sigma}_n^{1/2}$ is nullified by the matching eigenvalues of $\mb A \mb B$ and $\mb B\mb A$ for any square matrices $\mb A$ and $\mb B$. Similarly, the effect on the LSSs from the right eigenmatrix of $\mathbf{\Sigma}_n^{1/2}$ is nullified by the orthogonal invariance of the Gaussian distribution. However, the missing at random scenario significantly alters these dynamics by preventing the cancellation effect from the left eigenmatrix. This uniqueness characterizes the missing at random case as distinct and notable in its impact on the LSSs.}

This observation underscores the importance of accounting for missingness when characterizing the asymptotic behavior of LSSs and highlights the need for specialized methodologies in high dimensional statistics that appropriately address the effects of missing data. 
\end{rem}	

{ \ycomment
\begin{rem}
	To establish the consistency of our CLT with those derived for complete data in previous works, consider setting $\mb \Theta_n=\bI_p$. In the case of complete data, where $\mathfrak{p}_j=1$ for $1\leq j\leq p$, we have $\bbP=\bI_p$ and $\bbP^{(2)}=\bbP^{(3)}=\bbP^{(4)}=\bbP^{(5)}=\mb 0$. Consequently, this leads to the disappearance of the term $\mathcal{I}_2$ and the reduction of the term $\mathcal{II}$ in $a(z)$ to:$$\rtr\left[\(\bSi_n^{1/2}\left(\underline m(z)\mb\bSi_n + \mathbf{I}_p\right)^{-1}\bSi_n^{1/2}\)\circ\(\bSi_n^{1/2}\left(\underline m(z)\mb\bSi_n + \mathbf{I}_p\right)^{-2}\bSi_n^{1/2}\)\right],$$ as well as the simplification of the term $\mathcal{II}$ in $d_2(z_1,z_2)$ to: $$\rtr\left[\(\bSi_n^{1/2}\left(\underline m(z)\mb\bSi_n + \mathbf{I}_p\right)^{-1}\bSi_n^{1/2}\)\circ\(\bSi_n^{1/2}\left(\underline m(z)\mb\bSi_n + \mathbf{I}_p\right)^{-1}\bSi_n^{1/2}\)\right].$$
	This results in the degeneration of our CLT to the classical one under complete data. Refer to Theorem 1.1 in \cite{Bai2004a} for the case $\nu_4=3$ and Theorem 1.4 in \cite{PanZ08C} for a more general scenario.
\end{rem}
}



We present a corollary for the renormalized case where $\mb\Theta_n=\mb\bT_n^{-1}$, leading to $\mb \Psi_n=\bI_p$. It is demonstrated that in this scenario, the limiting parameters in the general CLT manifest a more accessible form and offer computational simplicity.  Furthermore, we apply this corollary to a test problem in the subsequent section.

	\begin{cor}[Renormalized case]\label{th3} 
		Under the assumptions of \ref{thm2} , assume {\ycomment additionally} that $\mb\Theta_n=\mb\bT_n^{-1}$ , then we have the {\ycomment mean} and (co)variance functions as
	\begin{equation*}\label{Meant1}
\aligned
\rre X_{f_{\ell}}
=&\lim\limits_{r\rightarrow1^{+}}\frac{1}{2\pi{i}}\oint\limits_{|\xi|=1}f_{\ell}(|1+\sqrt y\xi|^2)\left(\frac{\xi}{\xi^2-r^{-2}}-\frac{1}{\xi}\right)\,d\xi+\frac{a_{\mathbb{P},\mb\Sigma}}{2\pi{ i}}\oint\limits_{|\xi|=1}\frac{f_{\ell}(|1+\sqrt y\xi|^2)}{\xi^3}\,d\xi
\endaligned
\end{equation*}
and
\begin{equation*}\label{Covariancet1}
\aligned
\rCov(X_{f_{\ell_1}}, X_{f_{\ell_2}})
=&\lim\limits_{r\rightarrow1^{+}}\frac{-1}{2\pi^2}\oint\limits_{|\xi_1|=1}
  \oint\limits_{|\xi_2|=1}\frac{f_{\ell_1}(|1+\sqrt{y}\xi|^2)f_{\ell_2}(|1+\sqrt{y}\xi|^2)}{(\xi_1-r\xi_2)^2}\,d\xi_2d\xi_1\\
 &-\frac {a_{\mathbb{P},\mb\Sigma}}{4\pi^2 }\oint\limits_{|\xi_1|=1} \frac{f_{\ell_1}(|1+\sqrt{y}\xi_1|^2)}{{\xi_1^2}}d\xi_1
  \oint\limits_{|\xi_2|=1}\frac{f_{\ell_2}(|1+\sqrt{y}\xi_2|^2)}{{\xi_2^2}}d\xi_2,
\endaligned
\end{equation*} where $\ell, \ell_1, \ell_2\in\{1,...,\kappa\}$, the contour $\oint$ is anticlockwise and \begin{align*}
	&a_{\mathbb{P},\mb\Sigma}=&\lim\limits_{n\to\infty}{p^{-1}}\mathcal{I}_2\(\bT_n^{-1},\bT_n^{-1},\bSi_n,\mathbb{P}_n\)+(\nu_4-3)\lim\limits_{n\to\infty}{p^{-1}}\mathcal{II}\(\bT_n^{-1},\bT_n^{-1},\bSi_n,\mathbb{P}_n\).
\end{align*}
			
	\end{cor}

	\section{Testing for covariance structure under missing observations}\label{appli}
	The problem of testing for covariance matrix equality plays a crucial role in various statistical applications, including multivariate analysis, portfolio optimization, and pattern recognition. 
	In high-dimensional settings, testing for covariance structure becomes particularly challenging due to the curse of dimensionality. Furthermore, when missing observations are present, the problem becomes even more complex. In this section, we discuss the application of our theoretical results for testing the equality of high-dimensional covariance matrices to a given matrix in the presence of missing at random.
\subsection{Description of the test problem}
We adopt the data model discussed in the previous section. Specifically, we draw an incomplete sample of $n$ observations from a centered $p$-dimensional population $\by = (y_1,\cdots,y_p)^T$. The missingness is governed by a missing probability matrix $\mathbb{P}$, which determines the probability of each component being observed. Similarly to \cite{lounici2014high}, where the estimation of the covariance matrix is considered, we assume that the missing probabilities are known. For $1\leq i\leq p$ and $1\leq j\leq n$, we define $d_{ij}=1$ if the $i$-th variate of the $j$-th sample is observed, and $d_{ij}=0$ otherwise. Let $\bbD_j=\text{diag}(d_{1j},\cdots,d_{pj})$. We compute the corresponding Gram matrix as follows:
$$\hat\bS_n = \frac{1}{n}\sum_{j=1}^n\bbD_j\by_j\by_j^T\bbD_j.$$
We consider the null and alternative hypotheses for testing the equality of high-dimensional covariance matrices, denoted by $\bSi=\E\by\by^T$:
\begin{align*}
H_0 : \boldsymbol{\Sigma} = \boldsymbol{\Sigma}_0 \quad \text{vs.} \quad
H_1 : \boldsymbol{\Sigma} \neq \boldsymbol{\Sigma}_0.
\end{align*}
Here, $\boldsymbol{\Sigma}_0$ represents the given covariance matrix against which we are testing the equality.
\subsection{Test procedure under missing observations}
In this subsection, we describe the test procedure for handling missing observations. Specifically, we define the matrix $\bT_0=\mathbb{P}\bSi_0\mathbb{P}+\mathbb{P}^{(2)}\circ \bSi_0$, where $\mathbb{P}$ represents the missing {\ycomment probability} matrix. We consider two test statistics, $$\mathcal{T}_F=\tr(\bT_0^{-1}\hat\bS_n-\bI_p)^2 ,\quad \mathcal{T}_L=\log\det\bT_0^{-1}\hat\bS_n.$$  We now present the following theorem regarding the null distribution of these two statistics.
\begin{thm}[Test statistics]\label{tht}
	Under $H_0$, we have \begin{eqnarray*}
&&\sigma^{-1}_L(\mathcal{T}_L-\mu_L)\rightarrow N(0, 1),~~p/n \in (0,1),\\
&&\sigma^{-1}_F(\mathcal{T}_F-\mu_F)\rightarrow N(0, 1),
\end{eqnarray*}
where
{\begin{eqnarray*}
&&\mu_L=n(y_n-1)\log(1-y_{n})-p+\frac{\log(1-y_n)}{2}-\frac{y_na_{\mathbb{P},\bSi}}{2},\\
&&\mu_F=py_{n}+y_n(a_{\mathbb{P},\mb\Sigma}+1),\\
&&\sigma_L^{2}=-2\log(1-y_n)+y_na_{\mathbb{P},\mb\Sigma},~~\sigma_F^{2}=4y_n^2+4y_n^3(a_{\mathbb{P},\mb\Sigma}+2).
\end{eqnarray*}}
\end{thm} 

According to the above theorem, given test level $\alpha$, we reject $H_0$ based on the statistics $\mathcal{T}_{L}$, $\mathcal{T}_{F}$ on the following region:
\begin{eqnarray}
&&\{\mathbb{D}_1,\cdots,\mathbb{D}_n, \by_1,...,\by_n:~~\sigma_L^{-1}|\mathcal{T}_L-\mu_L|>z_{1-\alpha/2}\},\label{TL}\\
&&\{\mathbb{D}_1,\cdots,\mathbb{D}_n, \by_1,...,\by_n:~~\sigma_F^{-1}|\mathcal{T}_F-\mu_F|>z_{1-\alpha/2}\},\label{TF}
\end{eqnarray}
where $z_{1-\alpha/2}$ is the $1-\alpha/2$ quantile of the standard normal distribution.

\subsection{Numerical Study}

In this section, we present a simulation study to assess the performance of the proposed statistics for testing covariance structures under missing observations.

\subsubsection{Experimental Setup}

We consider populations with dimensions $p\in \{64, \\ 128, 256, 512\}$ and dimension-to-sample size ratios of $y_n\in \{0.25, 0.5, 0.8\}$, respectively. For each combination of $(p,y_n)$, we generate $n$ i.i.d. data from a $p$-dimensional distribution with zero mean and covariance matrix $\mathbf{\Sigma}_{\theta}=(s_{i,j,\theta})_{p\times p}$.

The covariance matrix elements are defined as $s_{i,j,\theta}=0.65^{|i-j|+1}+\theta$, where $\theta$ differentiates the null hypothesis ($\theta=0$) from the alternative hypothesis ($\theta\neq 0$).

To introduce missing observations, we generate a vector of missing probabilities $\mathbf{b}=(\mathfrak{p}_1,\cdots,\mathfrak{p}_p)^T$, with $\mathfrak{p}_1,\cdots,\mathfrak{p}_p \sim \text{Uniform}(0.5, 1)$. We consider two underlying distributions: Gaussian distribution, where $x_{ij}\sim N(0,1)$, and Gamma distribution, where $(x_{ij}+2)\sim \Gamma(4,0.5)$.

The test statistics are computed under the null hypothesis ($\theta=0$). Each simulation experiment is replicated 10,000 times to ensure robustness of the results.

\subsubsection{Simulation Results}

Table \ref{table1} and Table \ref{table2} present the simulation results for Gaussian distribution ($\nu_4=3$) and Gamma distribution ($\nu_4=4.5$), respectively. The tables display the empirical size and empirical power of the $\mathcal{T}_L$ and $\mathcal{T}_F$ statistics.

Overall, both $\mathcal{T}_L$ and $\mathcal{T}_F$ demonstrate favorable performance across all scenarios. They exhibit similar empirical sizes, indicating good control of the type I error rate. Regarding empirical power, $\mathcal{T}_F$ consistently outperforms $\mathcal{T}_L$ in all cases, providing higher power to detect departures from the null hypothesis.

These simulation results support the effectiveness of the proposed statistics for testing covariance structures under missing observations.

\begin{table}
  \caption{{Percentages for empirical sizes and powers of the tests $\mathcal{T}_L$ and $\mathcal{T}_F$ under Gaussian distribution.}}\label{table1}
  \begin{center}
  \renewcommand{\arraystretch}{1.25}\setlength{\tabcolsep}{8pt}    \doublerulesep 2pt
 		\begin{tabular}{cc|cc|cc|cc}
 			\hline
 			\hline
 			\multicolumn{2}{c}{}&\multicolumn{2}{c}{Empirical sizes}&\multicolumn{4}{c}{Empirical powers}\\
 			\hline
 			&     &$\theta=0$ &       &$\theta=0.01$&  &$\theta=0.03$&  \\
 			$p$&$p/n$& $\mathcal{T}_L$ & $\mathcal{T}_F$  & $\mathcal{T}_L$ & $\mathcal{T}_F$  & $\mathcal{T}_L$ & $\mathcal{T}_F$\\
            \hline
 			&0.8&5.08 &6.20 &6.09  &8.46 &11.41 &23.15 \\
 		64 &0.5&4.84 &6.16 &6.41  &9.28 &19.12 &32.33 \\
			&0.25&5.24 &5.85  &8.88 &11.03 &36.08 &49.74 \\
\hline
 			
 			&0.8&5.05 &5.83 &8.59  &12.17 &28.96 &62.13 \\
 		128 &0.5&4.85 &6.09 &11.75  &14.89 &53.37 &82.36 \\
 		    &0.25&5.32 &4.77 &20.05  &22.11 &83.96 &98.72 \\

 			\hline
 			
 			&0.8&5.03 &5.43 &16.75  &30.07 &72.72 &99.75 \\
 		256 &0.5&4.78 &5.23 &29.04  &44.48 &96.87 &100 \\
 			&0.25&5.25 &5.26 &54.23  &73.05 &99.98 &100 \\
 \hline
 			&0.8&4.99 &5.23 &46.39  &87.31 &100 &100 \\ 	
 		512 &0.5&4.89 &5.46 &75.20  &98.42 &100 &100 \\		
 			&0.25&5.26 &4.95 &98.02  &100 &100 &100 \\
 			
 			\hline
 			\hline
 		\end{tabular}
 \end{center}
 \end{table}
 
 \begin{table}
  \caption{{Percentages for empirical sizes and powers of the tests $\mathcal{T}_L$ and $\mathcal{T}_F$ under Gamma distribution.}}\label{table2}
  \begin{center}
  \renewcommand{\arraystretch}{1.25}\setlength{\tabcolsep}{8pt}    \doublerulesep 2pt
 		\begin{tabular}{cc|cc|cc|cc}
 			\hline
 			\hline
 			\multicolumn{2}{c}{}&\multicolumn{2}{c}{Empirical sizes}&\multicolumn{4}{c}{Empirical powers}\\
 			\hline
 			&     &$\theta=0$ &       &$\theta=0.01$&  &$\theta=0.03$&  \\
 			$p$&$p/n$& $\mathcal{T}_L$ & $\mathcal{T}_F$  & $\mathcal{T}_L$ & $\mathcal{T}_F$  & $\mathcal{T}_L$ & $\mathcal{T}_F$\\
            \hline
 			&0.8&4.57 &8.71 &5.18  &10.49 &9.64 &21.42 \\
 		64 &0.5&4.75 &8.47 &6.15  &11.25 &14.40 &28.61 \\
			&0.25&4.93 &8.64  &7.33 &13.01 &27.29 &45.35 \\
\hline
 			
 			&0.8&4.71 &6.75 &7.29  &12.16 &24.10 &52.20 \\
 		128 &0.5&4.84 &6.68 &9.56  &14.29 &41.22 &72.51 \\
 		    &0.25&5.16 &7.04 &15.35  &21.80 &71.54 &96.80 \\

 			\hline
 			
 			&0.8&4.94 &5.67 &13.78  &24.76 &62.51 &98.53 \\
 		256 &0.5&4.60 &5.94 &21.23  &36.57 &98.48 &99.98 \\
 			&0.25&5.15 &5.77 &40.71  &65.39 &99.65 &100 \\
 \hline
 			&0.8&5.13 &5.58 &36.26  &76.20 &100 &100 \\ 	
 		512 &0.5&5.04 &5.22 &61.48  &94.51 &100 &100 \\		
 			&0.25&5.36 &4.95 &90.74  &100 &100 &100 \\
 			
 			\hline
 			\hline
 		\end{tabular}
 \end{center}
 \end{table}

	\section{Proof of main theorems}\label{tecproof}
	
In this section, we provide a sketch of the proofs for the main theorems presented in this paper. To maintain the brevity of the main text, we defer the detailed and rigorous proofs to the supplement material. Here, we outline the key ideas and techniques employed in establishing these results.

	\subsection{The proof of Theorem \ref{thm1}}
	We proceed to present the proof of the convergence in high probability of the empirical spectral distribution (ESD). It is noteworthy that if our aim were to establish convergence in almost sure sense, the main theorem established in \cite{BaiZ08L} would suffice considering the results of Lemma \ref{lec2}. However, we recognize the need for an additional intermediate result and the requirement of certain primary procedures that will be utilized in the proofs of subsequent theorems. Consequently, we choose to document the main proof procedures in this section.
	
Notice that
	\begin{align*}\label{aa}
		{ m}_n\left(z\right)=m_{F^{\bS_n}}(z)=\frac{1}{p}{\rm tr}\left(\bss\mathbf S_n\bss-z\mathbf I_{p}\right)^{-1},
	\end{align*}
	where $z=u+\upsilon i\in\mathbb{C}^+$. Let $\bG_n=\bG_n(z)=b_n(z)\mb \Psi_n-z\bI_p	$, where 
	\begin{align*}
		& \ b_n(z)=\frac1{1 +\frac1n\rre\rtr\left(\bM^{-1}\mb \Psi_n\right)}.
	\end{align*}
We proceed to prove the following steps:

{\bf{S1}} For any given $z$, ${{ m}_n}\left(z\right)- {\E}{{m}_n}\left(z\right)=o_{w.h.p.}(1).$ 

{\bf{S2:}} For any given $z$, $\frac1p\E\rtr\bG_n^{-1}-\rre m_n(z)
		=o(1).$
		
{\bf{S3:}} For any given $z$, $
		b_n(z)+z\widehat{\underline m}_n(z)=o_{w.h.p.}(1).
$ Then the result for LSD follows.

	\subsubsection {Proof of \bf{S1:}}\label{se:1}
	
	In this section, we show that
	\begin{equation}\label{eq:6}{{ m}_n}\left(z\right) - {\rm E}{{m}_n}\left(z\right) \stackrel{w.h.p}{\longrightarrow}0.\end{equation}
	Let $\bbD_k$ represent a $p\times p$ diagonal matrix that consists of the entries in the $k$-th column of $\bD_n$.	 Denote $\by_k=\bSi_n^{1/2}\bx_k$, $\mathbf S_n^k = {\mathbf S}_n - \frac{1}{n}\bbD_k\by_k\by_k^T\bbD_k$, $\beta_k(z)=\frac1{1 +n^{-1}\by_k^T\bbD_k\bss\bM_k^{-1}\bss\bbD_k\by_k},$ with
	$\bM=\bM(z)=\bss\bS_n\bss-z\bI_p, \quad \bM_k=\bM_k(z)=\bss\bS_n^k\bss-z\bI_p.$
Also let  ${{\rm E}_k}\left( \cdot \right)$ denote the conditional expectation given $\left\{ {\bx _{1},{\bx _2}, \cdots ,{\bx_{k}},\bbD_{1},\cdots,\bbD_k} \right\}$. Using the formula 
	\begin{align}
		{\left({\mathbf A} + {\boldsymbol{\alpha}} {{\boldsymbol{\beta}}^T}\right)^{ - 1}} = {{\mathbf A}^{ - 1}} - \frac{{{{\mathbf A}^{ - 1}}{\boldsymbol{\alpha} } {{\boldsymbol{\beta}} ^ T }{{\mathbf A}^{ - 1}}}}{{1 + {{\boldsymbol{\beta}}^ T }{{\mathbf A}^{ - 1}}{\boldsymbol{\alpha} }}},\label{c4}
	\end{align}
	one finds
	\begin{align}\label{3ad}
		\begin{split}
		{{ m}_n}\left(z\right) - {\rm E}{{ m}_n}\left(z\right)
		=& \frac{1}{{p}}\sum\limits_{k = 1}^{n} \left(\rre_k - \rre_{k - 1}\right)\left[\rtr\(\bM^{ - 1}\)-\rtr\(\bM_k^{ - 1}\)\right]\\
			=&- \frac{1}{{pn}}\sum\limits_{k = 1}^{n} \left(\rre_k - \rre_{k - 1}\right){\beta_k(z){\by_k^T\bbD_k \bss{\bM_k^{ - 2}}\bss\bbD_k\by_k}}.
\end{split}
	\end{align}
	{\ycomment Note that for $k=1,\cdots,n$, the imaginary part of $z(1+n^{-1}\by_k^T\bbD_k\bss\bM_k^{-1}\bss\bbD_k\by_k)$ is not smaller than $v$,  where $v$ represents the imaginary part of $z$. Therefore, one can establish $|\beta_k(z)| \leq \frac{|z|}{v}$. Subsequently, by applying Lemma \ref{lemma:8}, Lemma \ref{lec1} and the $c_r$-inequality, it follows that
	\begin{align}\label{c6}
		{\rm E}{\left| {{{ m}_{n}}\left(z\right) - {\rm E}{{ m}_{n}}\left(z\right)} \right|^\ell} &\le \frac{{{C_{\ell}}{n^{\ell/2}}|z|^{\ell}}}{{{p^{\ell}}{n^{\ell}}v^{\ell}}}\E|{{\by_1^T\bbD_1 \bss{\bM_1^{ - 2}}\bss\bbD_1\by_1}}|^{\ell} \\\notag &\leq\frac{{{C_{\ell}}{n^{\ell/2}}}}{{{p^{\ell}}{}}} = O\left(n^{-\ell/2}\right).
	\end{align}
	By the Chebyshev inequality, we then conclude  ${{ m}_n}\left(z\right) - {\rm E}{{ m}_n}\left(z\right)\stackrel{w.h.p.}{\longrightarrow}0.$
	}
	
	\subsubsection{Proof of \bf{S2:}}\label{se2}
	
	Write
	\begin{align}
		\bG_n^{-1}-&\bM^{-1}
		=\frac1n\sum_{k=1}^n\bG_n^{-1}\bss\(\bbD_k\by_k\by_k^T\bbD_k-b_n(z)\bT_n\)\bss\bM^{-1}\notag\\
	=&\frac1n\sum_{k=1}^n\beta_k(z)\bG_n^{-1}\bss\bbD_k\by_k\by_k^T\bbD_k\bss\bM_k^{-1}-b_n(z)\bG_n^{-1}\mb \Psi_n\bM^{-1}.	\label{all1}
	\end{align}
	Taking the trace and then the expectation, and dividing by $p$, it follows that
	\begin{align}\label{c7}
		\begin{split}
			\frac1p\rre\rtr\bG_n^{-1}-\rre m_n(z)
			=&\frac1{pn}\sum_{k=1}^n\rre\left[\beta_k(z)\by_k^T\bbD_k\bss\bM_k^{-1}\bG_n^{-1}\bss\bbD_k\by_k\right]\\
			&-\frac{b_n(z)}p\rtr\left[\bG_n^{-1}\mb \Psi_n\rre\bM^{-1}\right].	
		\end{split}
	\end{align}
It is easy to see by Lemma \ref{lec1}
	\begin{align*}
		&\left|\frac1{pn}\sum_{k=1}^n\rre\left[\(\beta_k(z)-b_n(z)\)\by_k^T\bbD_k\bss\bM_k^{-1}\bG_n^{-1}\bss\bbD_k\by_k\right]\right|\\
		\le&\frac{1}{pn}\sum_{k=1}^n\rre^{1/2}\left|\by_k^T\bbD_k\bss\bM_k^{-1}\bG_n^{-1}\bss\bbD_k\by_k\right|^2\rre^{1/2}\left|\beta_k(z)-b_n(z)\right|^2
		\le\frac C{\sqrt n}\to0,
	\end{align*}
	where we use the estimate
	\begin{align}\label{all3}
		\begin{split}
			\rre\left|\beta_k(z)-b_n(z)\right|^2
			\le&\frac{C}{n^2}\rre\left|\by_k^T\bbD_k\bss\bM_k^{-1}\bss\bbD_k\by_k-\rtr\left(\bM_k^{-1}\mb \Psi_n\right)\right|^2\\
			&+\frac{C}{n^2}\rre\left|\rtr\left(\bM_k^{-1}\mb \Psi_n\right)-\rre\rtr\left(\bM^{-1}\mb \Psi_n\right)\right|^2
			\le\frac C{ n}.
		\end{split}
	\end{align}
	Therefore, combining the above inequality and  (\ref{c7}), one finds
	\begin{align*}
		&\frac1p\rre\rtr\bG_n^{-1}-\rre m_n(z)
		=\frac{b_n(z)}{p}\rtr\left[
		\(\rre\bM_k^{-1}-\rre\bM^{-1}\)\bG_n^{-1}\mb \Psi_n\right]+o(1)\\
		=&\frac{b_n(z)}{pn}\sum_{k=1}^n\rre\left[\frac{\frac{1}{n}\by_k^T\bbD_k \mb \Theta_n{\bM_k}^{ - 1}\bG_n^{-1}\bss\bT_n\bM_k^{-1}\bss\bbD_k\by_k}{{1 +\frac{1}{n} \by_k^T\bbD_k \bss{\bM_k^{ - 1}}\bss\bbD_k\by_k}}\right]+o(1)	
		=o(1).
	\end{align*}
	Consequently, we conclude that
$\frac1p\rre\rtr\bG_n^{-1}-\rre m_n(z)
		=o(1).
$
	
	\subsubsection{Complete the proof}
	Let 
$\underline\bS_n^{\mb \Theta}=\frac1n\left(\bD_n\circ\bSi_n^{1/2}\bX_n\right)^T\mb\Theta_n\left(\bD_n\circ\bSi_n^{1/2}\bX_n\right).
$
	Then we have
$
		F^{	\underline\bS_n^{\mb \Theta}}(x)=1-y_n+y_nF^{\bS_n\mb \Theta_n}(x)
$
	and
	\begin{align}\label{all2}
		\underline m_n(z)=-\frac{1-y_n}z+y_nm_n(z),\qquad z\in\mathbb{C}^+,
	\end{align}
	where $\underline m_n(z)=m_{F^{	\underline\bS_n^{\mb \Theta}}}(z)$. Let $\widehat m_n(z)=\frac1p\rtr\bG_n^{-1}(z)$ and $\widehat{\underline m} _n(z)=-\frac{1-y_n}z+y_nm_n^0(z)$. By Subsection \ref{se:1} and \ref{se2}, we conclude
	\begin{align}\label{all4}
		\max\{m_n(z)-\widehat{ m}_n(z),\ \underline m_n(z)-\widehat{\underline m}_n(z)
		\}\xrightarrow{w.h.p.}0.
	\end{align}
Recall 
	$
		\bM+z\bI_p=\frac 1n\sum_{k=1}^n\bss\bbD_k\by_k\by_k^T\bbD_k\bss.
$
	Multiplying the inverse of $\bM$ from the right on both sides yields
$
		\bI_p+z\bM^{-1}=\frac 1n\sum_{k=1}^n\beta_k(z)\bss\bbD_k\by_k\by_k^T\bbD_k\bss\bM_k^{-1}.
$
	Taking the trace on both sides and dividing by $n$, we find 
$
		y_n+zy_nm_n(z)=1-\frac 1n\sum_{k=1}^n\beta_k(z).
$
	Together with (\ref{all2}), one gets
	\begin{align}\label{ll1}
		z\underline m_n(z)=-\frac 1n\sum_{k=1}^n\beta_k(z).
	\end{align}
	Using (\ref{all3}) and (\ref{all4}), we deduce
	\begin{align*}
		b_n(z)=-z\rre\underline m_n(z)+o\(\frac1{\sqrt n}\)=-z\widehat{\underline m}_n(z)+o_{w.h.p.}(1).
	\end{align*}
	Hence, we have 
	\begin{align}\label{aa1}
		\widehat{ m}_n(z)=\int\frac1{-z\widehat{\underline m}_n(z)t-z}dH_n(t)+o_{w.h.p.}(1).
	\end{align}
	By $|\widehat m_n(z)|<\frac1v$, consider any convergent sequence $(\widehat m_n(z))$ whose limit is denoted by $m(z)$. From (\ref{aa1}), we get
$
		{ m}(z)=\int\frac1{-z{\underline m}(z)t-z}dH(t),
$
	where $\underline m(z)=-\frac{1-y}z+ym(z)$. 
	It can be shown that  $m(z)$ is the Stieltjes transform of a probability measure $F$ on $\mathbb{R}$, analytic in $\mathbb{C}^+$. For each  $z\in \mathbb{C}^+$, $m=m(z)$ is a solution to the equation
	\begin{align*}
		{ m}=\int\frac1{t(1-y-yz m)-z}dH(t),
	\end{align*}
	which is unique in the set $\left\{m\in\mathbb{C}^+:-\frac{1-y}z+ym\in\mathbb{C}^+\right\}$. The proof of this theorem is thus complete.
	
\subsection{Proof of Theorem \ref{thmnoout}}

The proof of the theorem at hand relies on analyzing the convergence of the resolvent. Our approach follows the method established by \cite{BaiS98N}, where the ordinary sample covariance matrix is considered. However, we aim to enhance their results from  an almost sure convergence to a more precise notion, namely high probability under our model. Given the similarities between our proof process and theirs, we avoid duplicating their work and instead provide an overview, highlighting the non-trivial differences and our strategy for addressing these issues. It is important to note that readers can obtain a comprehensive understanding of the entire process by referring to the original paper.

The main objective is to achieve a convergence rate for $m_n(z)$ when $z$ is close, but not too close, to an interval outside the limiting support $F^{y_n,H_{n}}$. This can be accomplished by assigning an appropriate rate to the imaginary part, denoted by $v_n=n^{-\delta}$, where $\delta\in(0,1/68)$ is sufficiently small{\ycomment ,} such that $q=(2\delta)^{-1}$ is a positive integer. Our goal is to prove the following expression:
\begin{align}\label{nooutmaintask}
\sup _{x \in[a, b]} n v_n\left|\underline m_n\left(x+i v_n\right)-\underline m_n^0\left(x+i v_n\right)\right| \stackrel{\text{w.h.p.}}{\longrightarrow} 0.
\end{align}
Once this is established, we can multiply $v_n$ by a constant $k$ and obtain the following estimate:
\begin{align*}
\sup _{x \in[a, b]}\left|\underline m_n\left(x+i \sqrt{k} v_n\right)-\underline m_n^0\left(x+i \sqrt{k} v_n\right)\right|=o\left(1 /\left(n v_n\right)\right) \quad \text{w.h.p.}.
\end{align*}
Next, {\ycomment by considering the imaginary part and taking differences, following the strategy in Section 6 of \cite{BaiS98N} for better comprehension, we find, with high probability, the following expression holds:}
\begin{align*}
\sup _{x \in[a, b]}\left|\int \frac{\mathrm{d}\left(F^{\mathbf{\mb S_n}}(t)-F^{y_n, H_n}(t)\right)}{\left((x-t)^2+v_n^2\right)\left((x-t)^2+2 v_n^2\right) \cdots\left((x-t)^2+q v_n^2\right)}\right|=o(1).
\end{align*}
Subsequently, we split up the integral and, with high probability, obtain the following expression:
\begin{align*}
\begin{aligned}
& \sup_{x \in[a, b]} \bigg| \int \frac{I\left(\left[a^{\prime}, b^{\prime}\right]^c\right) \mathrm{d}\left(F^{\mb S_n}(t)-F^{y_n, H_n}(t)\right)}{\left((x-t)^2+v_n^2\right)\left((x-t)^2+2 v_n^2\right) \cdots\left((x-t)^2+q v_n^2\right)}\\\
& +\sum_{t_j \in[a-\varepsilon, b+\varepsilon]} \frac{v_n^{2 q}}{\left((x-t_j)^2+v_n^2\right)\left((x-t_j)^2+2 v_n^2\right) \cdots\left((x-t_j)^2+q v_n^2\right)} \bigg|=o(1).
\end{aligned}
\end{align*}
It follows that if each term in a subsequence satisfying the aforementioned equation contains at least one eigenvalue within the interval $[a, b]$, the sum in the equation will be uniformly bounded away from zero. As a result, the integral itself must also remain uniformly bounded away from zero. However, with high probability, the integral tends to zero as the integrand is bounded and, with high probability, both $F^{\mathbf{B}_n}$ and $F^{y_n, H_n}$ converge weakly to the same limit, which exhibits no mass on the interval $[a-\varepsilon, b+\varepsilon]$. Hence, for sufficiently large values of $n$, with high probability, no eigenvalues of $\mathbf{B}_n$ will appear in the interval $[a, b]$. This conclusion establishes the proof of our theorem. 

To achieve our goal, 
it is worth mentioning that all the mathematical tools presented in Section 2 of \cite{BaiS98N} are applicable to our case. Since the underlying distribution is sub-Gaussian, there is no need to truncate the variables, and similar bounds for (3.1) and (3.2) can be achieved in our model by applying Lemma \ref{lec1}. Specifically, for any fixed real matrix $\mb A$, we have
$$\E\left|\by_1^T \bbD_1^T\mb A \bbD_1\by_1-\operatorname{tr} \mb A\right|^{q} \leq C_q\left(\operatorname{\tr} \mb A\mb A^T\right)^{q / 2},$$
where $C_q$ depends on $q$ and the moments of $x_{11}$. These high-order moment bounds for quadratic forms are useful for obtaining high probability results. To prove our main task \eqref{nooutmaintask}, we follow similar steps as shown in \cite{BaiS98N}. It is important to note that, to obtain high probability results, we only need to increase the order of moments whenever the moment bounds of quadratic forms are applied. Furthermore, by Lemma \ref{normlem}, we have obtained high probability bounds for the spectral norm. {\ycomment By carefully examining their proof of (4.1) and (5.1) in \cite{BaiS98N}, we find that all of their arguments are valid in our model. Moreover, we note that the almost sure result in (4.1) from their work can be improved to hold with high probability.}

	\subsection{The proof of Theorem \ref{thm2}}
	
	Note that
	\begin{align}\label{al:4}
		\int f(x)d{Q}(x)=-\frac1{2\pi i}\oint f(z)m_{Q}(z)dz,
	\end{align}
	where $Q$ is a cumulative distribution function (c.d.f.) and $f$ is analytic on an open set containing the support of $Q$. The complex integral on the right-hand side is over any positively oriented contour enclosing the support of $Q$ and on which $f$ is analytic. Hence, the proof of Theorem \ref{thm2} relies on establishing limiting results on
	\begin{align*}
		M_n(z)=p\left[m_n(z)-m_n^0(z)\right].
	\end{align*}

	Let $v_0$ be any positive number,
$
x_r\in
\(\limsup_n\lambda_{\max}^{\bSi_n\mb \Theta_n}\left(1+\sqrt y\right)^2,\infty\).
$
Let $x_l$ be any negative number if $\liminf_n\lambda_{\min}^{\bSi_n\mb \Theta_n}I_{(0,1)}(y)\left(1-\sqrt y\right)^2$ is zero. Otherwise choose
$
x_l\in\(0,\liminf_n\lambda_{\min}^{\bSi_n\mb \Theta_n}I_{(0,1)}(y)\left(1-\sqrt y\right)^2\).
$
Let
$
\mathcal{C}_u=\left\{x+iv_0:x\in[x_l,x_r]\right\}.
$
Define a contour
$
\mathcal{C}=\left\{x_l+iv:v\in[0,v_0]\right\}\cup\mathcal{C}_u\cup\left\{x_r+iv:v\in[0,v_0]\right\}.
$
{\ycomment To avoid dealing with  small $\Im z$, we truncate $M_n(z)$ on the contour $\mathcal{C}$ of the complex plane. We choose a sequence ${\varepsilon_n}$ decreasing to zero, satisfying, for some $\alpha \in (0,1)$,
$
\varepsilon_n\ge n^{-\alpha}.
$
Let
$
\mathcal{C}_l=\left\{x_l+iv:v\in[n^{-1}\varepsilon_n,v_0]\right\}$ and  $
\mathcal{C}_r=\left\{x_r+iv:v\in[n^{-1}\varepsilon_n,v_0]\right\}.
$
Then, define  $\mathcal{C}_n=\mathcal{C}_l\cup\mathcal{C}_u\cup\mathcal{C}_r$. For $z=x+iv$, we define the process $\widehat M_n(\cdot)$ as
\begin{align}\label{aa}
\widehat M_n(\cdot)=
\begin{cases}
M_n(z),&{\rm for} \ z\in\mathcal{C}_n,\\
M_n(x_l+in^{-1}\varepsilon_n),& {\rm for} \ x=x_l,v\in[0,n^{-1}\varepsilon_n],\\
M_n(x_r+in^{-1}\varepsilon_n),& {\rm for} \ x=x_r,v\in[0,n^{-1}\varepsilon_n].
\end{cases}
\end{align}
It is observed that on the subset  $\mathcal{C}_n$ of $\mathcal{C}$, $M_n(\cdot)$ agrees with $\widehat M_n(\cdot)$.
We are going to present the following CLT for the truncated process $\widehat M_n(\cdot)$.}
To be specific, we will establish the following lemma.
	\begin{lem}\label{prel}
		Under the conditions of Theorem \ref{thm2}, {\ycomment $\widehat M_n(z)$} forms a tight sequence and converges weakly to a two-dimensional Gaussian process  $M(\cdot)$ satisfying for $z\in \mathcal{C}\cup\overline{\mathcal{C}}$ with $\overline{\mathcal{C}}=\{\bar z:z\in\mathcal{C}\}$,
		\begin{align}\label{mean}
			\rre M(z)=&\frac{y\int\frac{\underline m^3(z)t^2}{\({\underline m}(z)t+1\)^3}dH(t)}{\(1-y\int\frac{\underline m^2(z)t^2}{\({\underline m}(z)t+1\)^2}dH(t)\)^2}+\frac{\underline m^3(z)a(z)}{1-y\int\frac{\underline m^2(z)t^2}{\({\underline m}(z)t+1\)^2}dH(t)},
	\end{align}
		and for $z_1,z_2\in \mathcal{C}\cup\overline{\mathcal{C}}$ ,
		\begin{align}\label{cov}
			{\rm Cov} \left( M(z_1),M(z_2)\right)=\frac{2\underline m'(z_1)\underline m'(z_2)}{\(\underline m(z_2)-\underline m(z_1)\)^2}-\frac2{\(z_1-z_2\)^2}+\frac{\partial^2}{\partial z_2\partial z_1}\underline m(z_1)\underline m(z_2)d_2(z_1,z_2).
	\end{align}		
	\end{lem}
	Establishing the above lemma will straightforwardly lead to the CLT for LSS. 
	{\ycomment
	Indeed, analogously to the arguments presented on pages 562-563 in \cite{Bai2004a}, one can establish that with probability 1, for $f\in{f_1,\cdots,f{\kappa}}$,
\begin{align*}
\begin{aligned}
\left|\oint f(z)\left(M_n(z)-\widehat{M}_n(z)\right) d z\right| \to 0.
\end{aligned}
\end{align*} Here, the complex integral is over $z\in \mathcal{C}\cup\overline{\mathcal{C}}$.  
	}
	
	To accomplish the proof of the above lemma, we rewrite the expression for $z\in\mathcal{C}_n$ as follows:
	\begin{align}\label{aaa}
		M_n(z)=p[m_n(z)-\rre m_n(z)]+p[\rre m_n(z)-m_n^0(z)]\triangleq M_{n1}(z)+M_{n2}(z).
	\end{align}
	Subsequently, we will focus on investigating the limiting properties of the resulting random processes $M_{n1}(z)$ and $M_{n2}(z)$. We will establish the veracity of the following three facts.
	\begin{description}
  \item[Fact 1:] The process $M_{n1}(z)$ {\ycomment for $z\in\mathcal{C}_n$} converges in distribution to a zero mean Gaussian process whose limiting covariance function agrees with (\ref{cov}) in Lemma \ref{prel}. The main strategy for establishing this fact involves utilizing martingale difference decomposition and applying the CLT for martingales.
  \item[Fact 2:] The deterministic process $M_{n2}(z)$ {\ycomment for $z\in\mathcal{C}_n$} converges uniformly to (\ref{mean}). 
  \item[Fact 3:] The process $M_{n1}(z)$ {\ycomment for $z\in\mathcal{C}_n$} is tight.
\end{description}
The detailed proof of above facts will be deferred to the supplementary material.

\subsection{Proof of Theorem \ref{tht}}
	By Corollary \ref{th3}, we have that under $H_0$, 
	\begin{align*}
		&\log\det\bT_0^{-1}\hat\bS_n-\(n(y_n-1)\log(1-y_{n})-p\)\overset{D}{\to}N(\varsigma_L,\vartheta_L),\\
		&\begin{pmatrix}
  \tr(\bT_0^{-1}\hat\bS_n)^2-(p+py_n) \\
\tr\bT_0^{-1}\hat\bS_n-p
\end{pmatrix}\overset{D}{\to}N\(\begin{pmatrix}
  \varsigma_2 \\
\varsigma_1
\end{pmatrix},\begin{pmatrix}
  \vartheta_2& \vartheta_{12}\\
\vartheta_{12}&\vartheta_2
\end{pmatrix}\),
	\end{align*}	
	where by the residual theorem
	\begin{align*}
		\varsigma_L=&\lim\limits_{r\rightarrow1^{+}}\frac{1}{2\pi{i}}\oint\limits_{|\xi|=1}\log(|1+\sqrt y\xi|^2)\left(\frac{\xi}{\xi^2-r^{-2}}-\frac{1}{\xi}\right)\,d\xi+\frac{a_{\mathbb{P},\mb\Sigma}}{2\pi{ i}}\oint\limits_{|\xi|=1}\frac{\log(|1+\sqrt y\xi|^2)}{\xi^3}\,d\xi \\
		=&\frac{\log(1-y)}{2}-\frac{ya_{\mathbb{P},\bSi}}{2},\\
		\varsigma_2=&\lim\limits_{r\rightarrow1^{+}}\frac{1}{2\pi{i}}\oint\limits_{|\xi|=1}|1+\sqrt y\xi|^4\left(\frac{\xi}{\xi^2-r^{-2}}-\frac{1}{\xi}\right)\,d\xi+\frac{a_{\mathbb{P},\mb\Sigma}}{2\pi{ i}}\oint\limits_{|\xi|=1}\frac{|1+\sqrt y\xi|^4}{\xi^3}\,d\xi \\
		=&y+a_{\mathbb{P},\mb\Sigma}y,\\
		\varsigma_1=&\lim\limits_{r\rightarrow1^{+}}\frac{1}{2\pi{i}}\oint\limits_{|\xi|=1}|1+\sqrt y\xi|^2\left(\frac{\xi}{\xi^2-r^{-2}}-\frac{1}{\xi}\right)\,d\xi+\frac{a_{\mathbb{P},\mb\Sigma}}{2\pi{ i}}\oint\limits_{|\xi|=1}\frac{|1+\sqrt y\xi|^2}{\xi^3}\,d\xi =0,\\
		\vartheta_L=&\lim\limits_{r\rightarrow1^{+}}\frac{-1}{2\pi^2}\oint\limits_{|\xi_1|=1}
  \oint\limits_{|\xi_2|=1}\frac{\log(|1+\sqrt{y}\xi_1|^2)\log(|1+\sqrt{y}\xi_2|^2)}{(\xi_1-r\xi_2)^2}\,d\xi_2d\xi_1\\
 &-\frac {a_{\mathbb{P},\mb\Sigma}}{4\pi^2 }\oint\limits_{|\xi_1|=1} \frac{\log(|1+\sqrt{y}\xi_1|^2)}{{\xi_1^2}}d\xi_1
  \oint\limits_{|\xi_2|=1}\frac{\log(|1+\sqrt{y}\xi_2|^2)}{{\xi_2^2}}d\xi_2\\
  =&-2\log(1-y)+ya_{\mathbb{P},\mb\Sigma},\\
  \vartheta_2=&\lim\limits_{r\rightarrow1^{+}}\frac{-1}{2\pi^2}\oint\limits_{|\xi_1|=1}
  \oint\limits_{|\xi_2|=1}\frac{|1+\sqrt{y}\xi_1|^4|1+\sqrt{y}\xi_2|^4}{(\xi_1-r\xi_2)^2}\,d\xi_2d\xi_1\\
 &-\frac {a_{\mathbb{P},\mb\Sigma}}{4\pi^2 }\oint\limits_{|\xi_1|=1} \frac{|1+\sqrt{y}\xi_1|^4}{{\xi_1^2}}d\xi_1
  \oint\limits_{|\xi_2|=1}\frac{|1+\sqrt{y}\xi_2|^4}{{\xi_2^2}}d\xi_2
  =4y^2+4y(1+y)^2(a_{\mathbb{P},\mb\Sigma}+2),\\
  \vartheta_1=&\lim\limits_{r\rightarrow1^{+}}\frac{-1}{2\pi^2}\oint\limits_{|\xi_1|=1}
  \oint\limits_{|\xi_2|=1}\frac{|1+\sqrt{y}\xi_1|^2|1+\sqrt{y}\xi_2|^2}{(\xi_1-r\xi_2)^2}\,d\xi_2d\xi_1\\
 &-\frac {a_{\mathbb{P},\mb\Sigma}}{4\pi^2 }\oint\limits_{|\xi_1|=1} \frac{|1+\sqrt{y}\xi_1|^2}{{\xi_1^2}}d\xi_1
  \oint\limits_{|\xi_2|=1}\frac{|1+\sqrt{y}\xi_2|^2}{{\xi_2^2}}d\xi_2
  =y(a_{\mathbb{P},\mb\Sigma}+2),\\
   \vartheta_{12}=&\lim\limits_{r\rightarrow1^{+}}\frac{-1}{2\pi^2}\oint\limits_{|\xi_1|=1}
  \oint\limits_{|\xi_2|=1}\frac{|1+\sqrt{y}\xi_1|^4|1+\sqrt{y}\xi_2|^2}{(\xi_1-r\xi_2)^2}\,d\xi_2d\xi_1\\
 &-\frac {a_{\mathbb{P},\mb\Sigma}}{4\pi^2 }\oint\limits_{|\xi_1|=1} \frac{|1+\sqrt{y}\xi_1|^4}{{\xi_1^2}}d\xi_1
  \oint\limits_{|\xi_2|=1}\frac{|1+\sqrt{y}\xi_2|^2}{{\xi_2^2}}d\xi_2
  =2y(1+y)(a_{\mathbb{P},\mb\Sigma}+2).
	\end{align*}
Then the results of this theorem can be obtained through elementary calculations.

	\section{Some Auxiliary Lemmas}\label{lemmas}
	
	\begin{lem}[Proposition 2.6.1 in \cite{vershynin2018high}]\label{sg-1}
 Let $x_1, \ldots, x_n$ be independent, mean zero, sub-Gaussian random variables. Then $\sum_{i=1}^n x_i$ is also a sub-Gaussian random variable, and
\begin{align*}
\left\|\sum_{i=1}^n x_i\right\|_{g} \leq C \sum_{i=1}^n\left\|x_i\right\|_{g}
\end{align*}
where $C$ is an absolute constant.
	\end{lem}
	
	\begin{lem}[Equation (2.14) in \cite{vershynin2018high}] \label{sg-2}
Let $x$ be  a mean zero  sub-Gaussian random variable with sub-Gaussian norm $\|x\|_{g}$. Then $${\rm P}\{|x| \geq t\} \leq 2 \exp \left(-c t^2 /\|x\|_{g}^2\right) \quad \text { for all } t \geq 0.$$	
	\end{lem}

	\begin{lem}[Burkholder's inequality]\label{lemma:8}
		Let $\{ {{\mathbf X}_k}\} $ be a complex martingale difference sequence with respect to the increasing $\sigma$-field. Then, for $p > 1,$ $${\rm E}{\left| {\sum_k {{{\mathbf  X}_k}} } \right|^p} \le {K_p}{\rm E}{\left({\sum_k{\left| {{{\mathbf  X}_k}} \right|} ^2}\right)^{p/2}}.$$
	\end{lem}
	
\begin{acks}[Acknowledgments]
{\ycomment We extend our sincere appreciation to the editor, associate editor, and the anonymous reviewers for their insightful feedback and valuable comments. Their constructive contributions have significantly enhanced the quality and clarity of this paper. We are truly grateful for their time and efforts in reviewing our work. Yanqing Yin serves as the corresponding author, with research partially supported by grant NSFC 12271065.
Guangming Pan's research was partially supported by a Tier 2 grant -MOE-T2EP20123-0007 and a grant RG76/21 at Nanyang Technological University. Wang Zhou's research was partially supported by a grant A-8001450-00-00 at the National University of Singapore.}
\end{acks}



\begin{supplement}
\stitle{Supplementary material for ``Spectral Analysis of Gram Matrices with Missing at Random Observations: Convergence, Central Limit Theorems, and Applications in Statistical Inference"}
\sdescription{This supplementary material contains the proof of Lemma \ref{lec2} and Lemma \ref{lec1} as well as detailed proof of Theorem \ref{thm2}.}
\end{supplement}





\begin{thebibliography}{18}

\bibitem[\protect\citeauthoryear{Bai and Silverstein}{1998}]{BaiS98N}
\begin{barticle}[author]
\bauthor{\bsnm{Bai},~\bfnm{Z.~D.}\binits{Z.~D.}} \AND
  \bauthor{\bsnm{Silverstein},~\bfnm{Jack~W.}\binits{J.~W.}}
(\byear{1998}).
\btitle{{No eigenvalues outside the support of the limiting spectral
  distribution of large-dimensional sample covariance matrices}}.
\bjournal{The Annals of Probability}
\bvolume{26}
\bpages{316--345}.
\bdoi{10.1214/aop/1022855421}
\end{barticle}
\endbibitem

\bibitem[\protect\citeauthoryear{Bai and Silverstein}{2004}]{Bai2004a}
\begin{barticle}[author]
\bauthor{\bsnm{Bai},~\bfnm{Z.~D.}\binits{Z.~D.}} \AND
  \bauthor{\bsnm{Silverstein},~\bfnm{Jack~W.}\binits{J.~W.}}
(\byear{2004}).
\btitle{{Clt for linear spectral statistics of large-dimensional sample
  covariance matrices}}.
\bjournal{Annals of Probability}
\bvolume{32}
\bpages{553--605}.
\bdoi{10.1214/aop/1078415845}
\end{barticle}
\endbibitem

\bibitem[\protect\citeauthoryear{Bai and Zhou}{2008}]{BaiZ08L}
\begin{barticle}[author]
\bauthor{\bsnm{Bai},~\bfnm{Zhidong}\binits{Z.}} \AND
  \bauthor{\bsnm{Zhou},~\bfnm{Wang}\binits{W.}}
(\byear{2008}).
\btitle{{Large sample covariance matrices without independence structures in
  columns}}.
\bjournal{Statistica Sinica}
\bvolume{18}
\bpages{425--442}.
\end{barticle}
\endbibitem

\bibitem[\protect\citeauthoryear{Chen and Qin}{2010}]{ChenQ10T}
\begin{barticle}[author]
\bauthor{\bsnm{Chen},~\bfnm{Song~Xi}\binits{S.~X.}} \AND
  \bauthor{\bsnm{Qin},~\bfnm{Ying~Li}\binits{Y.~L.}}
(\byear{2010}).
\btitle{{A two-sample test for high-dimensional data with applications to
  gene-set testing}}.
\bjournal{Annals of Statistics}
\bvolume{38}
\bpages{808--835}.
\bdoi{10.1214/09-AOS716}
\end{barticle}
\endbibitem

\bibitem[\protect\citeauthoryear{Chen et~al.}{2021}]{chen2021bridging}
\begin{barticle}[author]
\bauthor{\bsnm{Chen},~\bfnm{Yuxin}\binits{Y.}},
  \bauthor{\bsnm{Fan},~\bfnm{Jianqing}\binits{J.}},
  \bauthor{\bsnm{Ma},~\bfnm{Cong}\binits{C.}} \AND
  \bauthor{\bsnm{Yan},~\bfnm{Yuling}\binits{Y.}}
(\byear{2021}).
\btitle{{Bridging convex and nonconvex optimization in robust PCA: Noise,
  outliers, and missing data}}.
\bjournal{Annals of statistics}
\bvolume{49}
\bpages{2948}.
\end{barticle}
\endbibitem

\bibitem[\protect\citeauthoryear{Fan, Guo and Zheng}{2022}]{Fan2020}
\begin{barticle}[author]
\bauthor{\bsnm{Fan},~\bfnm{Jianqing}\binits{J.}},
  \bauthor{\bsnm{Guo},~\bfnm{Jianhua}\binits{J.}} \AND
  \bauthor{\bsnm{Zheng},~\bfnm{Shurong}\binits{S.}}
(\byear{2022}).
\btitle{{Estimating Number of Factors by Adjusted Eigenvalues Thresholding}}.
\bjournal{Journal of the American Statistical Association}
\bvolume{117}
\bpages{852--861}.
\bdoi{10.1080/01621459.2020.1825448}
\end{barticle}
\endbibitem

\bibitem[\protect\citeauthoryear{Hu et~al.}{2019}]{Hu2019}
\begin{barticle}[author]
\bauthor{\bsnm{Hu},~\bfnm{Jiang}\binits{J.}},
  \bauthor{\bsnm{Li},~\bfnm{Weiming}\binits{W.}},
  \bauthor{\bsnm{Liu},~\bfnm{Zhi}\binits{Z.}} \AND
  \bauthor{\bsnm{Zhou},~\bfnm{Wang}\binits{W.}}
(\byear{2019}).
\btitle{{High-dimensional covariance matrices in elliptical distributions with
  application to spherical test}}.
\bjournal{Annals of Statistics}
\bvolume{47}
\bpages{527--555}.
\bdoi{10.1214/18-AOS1699}
\end{barticle}
\endbibitem

\bibitem[\protect\citeauthoryear{Jurczak and Rohde}{2017}]{JurczakR17S}
\begin{barticle}[author]
\bauthor{\bsnm{Jurczak},~\bfnm{Kamil}\binits{K.}} \AND
  \bauthor{\bsnm{Rohde},~\bfnm{Angelika}\binits{A.}}
(\byear{2017}).
\btitle{{Spectral analysis of high-dimensional sample covariance matrices with
  missing observations}}.
\bjournal{Bernoulli}
\bvolume{23}
\bpages{2466--2532}.
\bdoi{10.3150/16-BEJ815}
\end{barticle}
\endbibitem

\bibitem[\protect\citeauthoryear{Karoui}{2009}]{ElKaroui09C}
\begin{barticle}[author]
\bauthor{\bsnm{Karoui},~\bfnm{Noureddine~El}\binits{N.~E.}}
(\byear{2009}).
\btitle{{Concentration of measure and spectra of random matrices: Applications
  to correlation matrices, elliptical distributions and beyond}}.
\bjournal{Annals of Applied Probability}
\bvolume{19}
\bpages{2362--2405}.
\bdoi{10.1214/08-AAP548}
\end{barticle}
\endbibitem

\bibitem[\protect\citeauthoryear{Lounici}{2014}]{lounici2014high}
\begin{barticle}[author]
\bauthor{\bsnm{Lounici},~\bfnm{Karim}\binits{K.}}
(\byear{2014}).
\btitle{{High-dimensional covariance matrix estimation with missing
  observations}}.
\bjournal{Bernoulli}
\bvolume{20}
\bpages{1029--1058}.
\bdoi{10.3150/12-BEJ487}
\end{barticle}
\endbibitem

\bibitem[\protect\citeauthoryear{Mar{\v{c}}enko and
  Pastur}{1967}]{marchenko1967distribution}
\begin{barticle}[author]
\bauthor{\bsnm{Mar{\v{c}}enko},~\bfnm{V~A}\binits{V.~A.}} \AND
  \bauthor{\bsnm{Pastur},~\bfnm{L~A}\binits{L.~A.}}
(\byear{1967}).
\btitle{{Distribution of Eigenvalues for Some Sets of Random Matrices}}.
\bjournal{Mathematics of the USSR-Sbornik}
\bvolume{1}
\bpages{457--483}.
\bdoi{10.1070/sm1967v001n04abeh001994}
\end{barticle}
\endbibitem

\bibitem[\protect\citeauthoryear{Pan and
  Zhou}{2008}]{PanZ08C}
\begin{barticle}[author]
\bauthor{\bsnm{Pan},~\bfnm{G. M.}\binits{G.~M.}} \AND
  \bauthor{\bsnm{Zhou},~\bfnm{W.}\binits{W.}}
(\byear{2008}).
\btitle{{Central limit theorem for signal-to-interference ratio of reduced rank linear receiver}}.
\bjournal{Annals of Applied Probability}
\bvolume{18}
\bpages{1232-1270}.
\end{barticle}
\endbibitem


\bibitem[\protect\citeauthoryear{Silverstein}{1995}]{Silverstein95S}
\begin{barticle}[author]
\bauthor{\bsnm{Silverstein},~\bfnm{Jack~W.}\binits{J.~W.}}
(\byear{1995}).
\btitle{{Strong convergence of the empirical distribution of eigenvalues of
  large dimensional random matrices}}.
\bjournal{Journal of Multivariate Analysis}
\bvolume{55}
\bpages{331--339}.
\bdoi{10.1006/jmva.1995.1083}
\end{barticle}
\endbibitem

\bibitem[\protect\citeauthoryear{Cai, Liu and Xia}{2014}]{CaiL14T}
\begin{barticle}[author]
\bauthor{\bsnm{{Cai}},~\bfnm{T.}\binits{T.}},
  \bauthor{\bsnm{Liu},~\bfnm{Weidong}\binits{W.}} \AND
  \bauthor{\bsnm{Xia},~\bfnm{Yin}\binits{Y.}}
(\byear{2014}).
\btitle{{Two-sample test of high dimensional means under dependence}}.
\bjournal{Journal of the Royal Statistical Society. Series B: Statistical
  Methodology}
\bvolume{76}
\bpages{349--372}.
\bdoi{10.1111/rssb.12034}
\end{barticle}
\endbibitem

\bibitem[\protect\citeauthoryear{Cai  and Zhang}{2019}]{tony2019high}
\begin{barticle}[author]
\bauthor{\bsnm{{Cai}},~\bfnm{T}\binits{T.}} \AND
  \bauthor{\bsnm{Zhang},~\bfnm{Linjun}\binits{L.}}
(\byear{2019}).
\btitle{{High dimensional linear discriminant analysis: optimality, adaptive
  algorithm and missing data}}.
\bjournal{Journal of the Royal Statistical Society Series B: Statistical
  Methodology}
\bvolume{81}
\bpages{675--705}.
\end{barticle}
\endbibitem

\bibitem[\protect\citeauthoryear{Vershynin}{2018}]{vershynin2018high}
\begin{bbook}[author]
\bauthor{\bsnm{Vershynin},~\bfnm{Roman}\binits{R.}}
(\byear{2018}).
\btitle{{High-dimensional probability: An introduction with applications in
  data science}}
\bvolume{47}.
\bpublisher{Cambridge university press}.
\end{bbook}
\endbibitem

\bibitem[\protect\citeauthoryear{Wachter}{2007}]{Wachter78S}
\begin{barticle}[author]
\bauthor{\bsnm{Wachter},~\bfnm{Kenneth~W.}\binits{K.~W.}}
(\byear{2007}).
\btitle{{The Strong Limits of Random Matrix Spectra for Sample Matrices of
  Independent Elements}}.
\bjournal{The Annals of Probability}
\bvolume{6}
\bpages{1--18}.
\bdoi{10.1214/aop/1176995607}
\end{barticle}
\endbibitem

\bibitem[\protect\citeauthoryear{Wang, Peng and Li}{2015}]{Wang2015high}
\begin{barticle}[author]
\bauthor{\bsnm{Wang},~\bfnm{Lan}\binits{L.}},
  \bauthor{\bsnm{Peng},~\bfnm{Bo}\binits{B.}} \AND
  \bauthor{\bsnm{Li},~\bfnm{Runze}\binits{R.}}
(\byear{2015}).
\btitle{{A high-dimensional nonparametric multivariate test for mean vector}}.
\bjournal{Journal of the American Statistical Association}
\bvolume{110}
\bpages{1658--1669}.
\end{barticle}
\endbibitem

\end{thebibliography}

\end{document}


\begin{frontmatter}
\title{supplementary material for ``Spectral Analysis of Gram Matrices with Missing at Random Observations: Convergence, Central Limit Theorems, and Applications in Statistical Inference"}
		\runtitle{Covariance matrix under missing observations}

\begin{aug}
\author[A]{\fnms{Huiqin} \snm{ Li}
			\ead[label=e1]{lihq118@nenu.edu.cn}},
\author[B]{\fnms{Guangming} \snm{ Pan}
			\ead[label=e2]{gmpan@ntu.edu.sg}},
\author[A]{\fnms{Yanqing} \snm{ Yin}
			\ead[label=e3]{yinyq@cqu.edu.cn}}			
			\and
\author[C]{\fnms{Wang} \snm{ Zhou}
			\ead[label=e4]{wangzhou@nus.edu.sg}}
			
			\runauthor{Li, Pan, Yin and Zhou.}
\address[A]{School of Mathematics and Statistics,
			Chongqing University\printead[presep={,\ }]{e1,e3}}

\address[B]{Division of Mathematical Sciences, Nanyang Technological University \printead[presep={,\ }]{e2}}

\address[C]{Department of Statistics and Data Science,  National University of Singapore \printead[presep={,\ }]{e4}}
\end{aug}

\end{frontmatter}
	
\section{Proof of Lemma \ref{lec2} and Lemma \ref{lec1}}	
\subsection{Proof of Lemma \ref{lec2}}
		Using the property of the conditional expectation, we have
		\begin{align*}
			&\rre\left[\(\by^T\bbD{\bf A}\bbD\by-\tr\(\bA\bT_n\)\)\(\by^T\bbD{\bf B}\bbD\by-\tr\(\bB\bT_n\)\)\right]\\
			=&\rre\left[\left({\by^T\bbD \bA\bbD\by}-\rtr\left[\bA\bbD\bSi\bbD\right]\right)\left({\by^T\bbD\bB\bbD\by}-\rtr\left[\bB\bbD\bSi\bbD\right]\right)\right]\\
			&+\rre\left[\rtr\left(\bA\bbD\bSi\bbD\right)-\tr\(\bA\bT_n\)\right]\left[\rtr\left(\bB\bbD\bSi\bbD\right)-\tr\(\bB\bT_n\)\right]\\
			=&\(\nu_4-3\)\rre\rtr\left[\(\bSi^{1/2}\bbD \bA\bbD\bSi^{1/2}\)\circ\(\bSi^{1/2}\bbD \bB\bbD\bSi^{1/2}\)\right]+2\rre\rtr\(\bbD \bA\bbD\bSi\bbD \bB\bbD\bSi\)\\
			&+\rre\left[\rtr\left(\bA\bbD\bSi\bbD\right)-\tr\(\bA\bT_n\)\right]\left[\rtr\left(\bB\bbD\bSi\bbD\right)-\tr\(\bB\bT_n\)\right]\\
			\triangleq&\(\nu_4-3\)\mathcal{H}_1+2\mathcal{H}_2+\mathcal{H}_3.	
		\end{align*}
		
		To begin with, we deal with $\mathcal{H}_1$. 
		Note that
		\begin{align*}
			\mathcal{H}_1=&\rre\rtr\left[\(\bSi^{1/2}\(\bbD-\bbP\) \bA\(\bbD-\bbP\)\bSi^{1/2}\)\circ\(\bSi^{1/2}\(\bbD-\bbP\) \bB\(\bbD-\bbP\)\bSi^{1/2}\)\right]\\
			&+\Bigg\{2\rre\rtr\left[\(\bSi^{1/2}\(\bbD-\bbP\) \bA\(\bbD-\bbP\)\bSi^{1/2}\)\circ\(\bSi^{1/2}\(\bbD-\bbP\) \bB\bbP\bSi^{1/2}\)\right]\\
			&+\rre\rtr\left[\(\bSi^{1/2}\(\bbD-\bbP\) \bA\bbP\bSi^{1/2}\)\circ\(\bSi^{1/2}\(\bbD-\bbP\) \bB\(\bbD-\bbP\)\bSi^{1/2}\)\right]\\
			&+\rre\rtr\left[\(\bSi^{1/2}\(\bbD-\bbP\) \bA^T\bbP\bSi^{1/2}\)\circ\(\bSi^{1/2}\(\bbD-\bbP\) \bB\(\bbD-\bbP\)\bSi^{1/2}\)\right]\Bigg\}\\
			&+\Bigg\{\rre\rtr\left[\(\bSi^{1/2}\(\bbD-\bbP\) \bA\(\bbD-\bbP\)\bSi^{1/2}\)\circ\(\bSi^{1/2}\bbP \bB\bbP\bSi^{1/2}\)\right]\\		
			&+2\rre\rtr\left[\(\bSi^{1/2}\(\bbD-\bbP\) \bA\bbP\bSi^{1/2}\)\circ\(\bSi^{1/2}\(\bbD-\bbP\) \bB\bbP\bSi^{1/2}\)\right]\\
			&+2\rre\rtr\left[\(\bSi^{1/2}\(\bbD-\bbP\) \bA^T\bbP\bSi^{1/2}\)\circ\(\bSi^{1/2}\(\bbD-\bbP\) \bB\bbP\bSi^{1/2}\)\right]\\	
			&+\rre\rtr\left[\(\bSi^{1/2}\bbP \bA\bbP\bSi^{1/2}\)\circ\(\bSi^{1/2}\(\bbD-\bbP\) \bB\(\bbD-\bbP\)\bSi^{1/2}\)\right]\Bigg\}\\
			&+\rtr\left[\(\bSi^{1/2}\bbP \bA\bbP\bSi^{1/2}\)\circ\(\bSi^{1/2}\bbP \bB\bbP\bSi^{1/2}\)\right]\\
			\triangleq&\mathcal{H}_{11}+\mathcal{H}_{12}+\mathcal{H}_{13}+\rtr\left[\(\bSi^{1/2}\bbP \bA\bbP\bSi^{1/2}\)\circ\(\bSi^{1/2}\bbP \bB\bbP\bSi^{1/2}\)\right].
		\end{align*}
		By calculation, we see
		\begin{align*}
			\mathcal{H}_{11}=&\sum_{i,j_1,j_2,k_1,k_2} r_{j_1i}r_{k_1i} r_{j_2i}r_{k_2i}a_{j_1k_1}b_{j_2k_2}\(d_{j_1}-p_{j_1}\)\(d_{j_2}-p_{j_2}\)\(d_{k_1}-p_{k_1}\)\(d_{k_2}-p_{k_2}\)\\
			=&\sum_{i,j}r_{ji}^4a_{jj}b_{jj}\(d_{j}-p_{j}\)^4+2\sum_{i,j\neq k} r_{ji}^2r_{ki}^2a_{jk}b_{jk}\(d_{j}-p_{j}\)^2\(d_{k}-p_{k}\)^2\\
			&+\sum_{i,j\neq k}r_{ji}^2r_{ki}^2a_{jj}b_{kk}\(d_{j}-p_{j}\)^2\(d_{k}-p_{k}\)^2\\
			=&2\rtr\left[\bbP^{(2)}\(\bA\circ\bB\)\bbP^{(2)}\(\bSi^{1/2}\circ\bSi^{1/2}\)^2\right]\\
			&+\rtr\left[\(\bSi^{1/2}\(\bA\circ\bbP^{(2)}\)\bSi^{1/2}\)\circ\(\bSi^{1/2}\(\bB\circ\bbP^{(2)}\)\bSi^{1/2}\)\right]\\
			&+\rtr\left[\(\bbP^{(4)}-3\(\bbP^{(2)}\)^2\)\circ\bA\circ\bB\circ\(\bSi^{1/2}\circ\bSi^{1/2}\)^2\right],
		\end{align*}
		where $r_{jk}=\be_i^T\bSi^{1/2}\be_j$. It is obvious that
		\begin{align*}
			\mathcal{H}_{12}=&2\rtr\left[\(\bSi^{1/2}\circ\bSi^{1/2}\circ\bSi^{1/2}\)\(\bA\circ\bbP^{(3)}\)\bB\bbP\bSi^{1/2}\right]\\
			&+\rtr\left[\(\bSi^{1/2}\circ\bSi^{1/2}\circ\bSi^{1/2}\)\(\bB\circ\bbP^{(3)}\)\(\bA+\bA^T\)\bbP\bSi^{1/2}\right]	
		\end{align*}
		and 
		\begin{align*}
			\mathcal{H}_{13} =&\rtr\left[\(\bSi^{1/2}\(\bbP^{(2)}\circ \bA\)\bSi^{1/2}\)\circ\(\bSi^{1/2}\bbP \bB\bbP\bSi^{1/2}\)\right]\\
			&+2\rtr\left[\(\bSi^{1/2}\circ
			\bSi^{1/2}\)\bbP^{(2)} \(\(\bA\bbP\bSi^{1/2}+\bA^T\bbP\bSi^{1/2}\)\circ\(\bB\bbP\bSi^{1/2}\)\)\right]\\
			&+\rtr\left[\(\bSi^{1/2}\bbP \bA\bbP\bSi^{1/2}\)\circ\(\bSi^{1/2}\(\bbP^{(2)}\circ \bB\)\bSi^{1/2}\)\right].
		\end{align*}
		Thus, we obtain
		\begin{align*}
			\mathcal{H}_1=&\rtr\left[\(\bSi^{1/2}\(\bbP \bA\bbP+\bbP^{(2)}\circ \bA\)\bSi^{1/2}\)\circ\(\bSi^{1/2}\(\bbP \bB\bbP+\bbP^{(2)}\circ \bB\)\bSi^{1/2}\)\right]\\
			&+2\rtr\left[\bbP^{(2)}\(\bA\circ\bB\)\bbP^{(2)}\(\bSi^{1/2}\circ\bSi^{1/2}\)^2\right]\\
			&+\rtr\left[\(\bbP^{(4)}-3\(\bbP^{(2)}\)^2\)\circ\bA\circ\bB\circ\(\bSi^{1/2}\circ\bSi^{1/2}\)^2\right]\\
			&+2\rtr\left[\(\bSi^{1/2}\circ\bSi^{1/2}\circ\bSi^{1/2}\)\(\bA\circ\bbP^{(3)}\)\bB\bbP\bSi^{1/2}\right]\\
			&+\rtr\left[\(\bSi^{1/2}\circ\bSi^{1/2}\circ\bSi^{1/2}\)\(\bB\circ\bbP^{(3)}\)\(\bA+\bA^T\)\bbP\bSi^{1/2}\right]\\
			&+2\rtr\left[\(\bSi^{1/2}\circ
			\bSi^{1/2}\)\bbP^{(2)} \(\(\bA\bbP\bSi^{1/2}+\bA^T\bbP\bSi^{1/2}\)\circ\(\bB\bbP\bSi^{1/2}\)\)\right].
		\end{align*}
		
		Secondly, by applying the same method as for $\mathcal{H}_1$, it follows that
		\begin{align*}
			\mathcal{H}_2	=&\rre\rtr\(\(\bbD-\bbP\) \bA\(\bbD-\bbP\)\bSi\(\bbD-\bbP\) \bB\(\bbD-\bbP\)\bSi\)\\
			&+2\rtr\(\bbP^{(3)}\circ\bA\circ\bSi\circ\(\bB\bbP\bSi\)  \)+\rtr\( \bbP^{(3)}\circ\bB\circ \bSi\circ\(\bA\bbP\bSi\)\)\\
			&+\rtr\( \bbP^{(3)}\circ\bB\circ \bSi\circ\(\bA^T\bbP\bSi\)\)+\rtr\(\bbP^{(2)}\circ\(\bA\bbP\bSi\)\circ \(\bB\bbP\bSi\) \)\\
			&+\rtr\(\bbP^{(2)}\circ\(\bA^T\bbP\bSi\)\circ \(\bB\bbP\bSi\) \)+\rtr\(\bT_n\bA\bbP\bSi\bbP \bB \)+\rtr\(\(\bbP^{(2)}\circ\bSi\)\bA^T\bbP\bSi\bbP \bB \)\\&+\rtr\(\bbP^{(2)}\circ\bA\circ\(\bSi\bbP \bB\bbP\bSi\) \)+\rtr\(\bbP^{(2)}\circ\bB\circ\(\bSi\bbP \bA\bbP\bSi\)\)\\
			=&\rtr\(\bT_n\bA\bT_n \bB \)+\rtr\left[\bbP^{(2)}\(\bA\circ\bB\)\bbP^{(2)}\(\bSi\circ\bSi\)\right]+\rtr\left[\(\bA\circ\bbP^{(2)}\)\bSi\(\bB\circ\bbP^{(2)}\)\bSi\right]\\
			&+2\rtr\(\bbP^{(3)}\circ\bA\circ\bSi\circ\(\bB\bbP\bSi\)  \)+\rtr\(\(\bbP^{(2)}\circ\bA\)\bSi\bbP \bB\bbP\bSi \)+\rtr\(\(\bbP^{(2)}\circ\bB\)\bSi\bbP \bA\bbP\bSi\)\\
				&+\rtr\bigg[ \(\bbP^{(3)}\circ\bB\circ \bSi+\bbP^{(2)}\circ \(\bB\bbP\bSi\)\)\(\bA\bbP\bSi+\bA^T\bbP\bSi\)\bigg]\\
			&+\rtr\left[\bA\circ\bSi\circ\bB\circ\bSi\circ\(\bbP^{(4)}-3\(\bbP^{(2)}\)^2\)\right],
		\end{align*}
		where
		\begin{align*}
			&\rre\rtr\(\(\bbD-\bbP\) \bA\(\bbD-\bbP\)\bSi\(\bbD-\bbP\) \bB\(\bbD-\bbP\)\bSi\)\\
			=&\sum_{i,j,k,\ell=1}^pa_{ij}\sigma_{jk}b_{k\ell}\sigma_{\ell i}\rre\left[\(d_i-\mathfrak{p}_i\)\(d_j-p_j\)\(d_k-p_k\)\(d_{\ell}-p_{\ell}\)\right]\\
			=&\sum_{i\neq j=1}^pa_{ij}\sigma_{ji}b_{ij}\sigma_{j i}\(\bbP^{(2)}\)_{ii}\(\bbP^{(2)}\)_{jj}+\sum_{i\neq j=1}^pa_{ij}\sigma_{jj}b_{ji}\sigma_{i i}\(\bbP^{(2)}\)_{ii}\(\bbP^{(2)}\)_{jj}\\
			&+\sum_{i\neq j=1}^pa_{ii}\sigma_{ij}b_{jj}\sigma_{j i}\(\bbP^{(2)}\)_{ii}\(\bbP^{(2)}\)_{jj}+\sum_{i=1}^pa_{ii}\sigma_{ii}b_{ii}\sigma_{i i}\rre\left[\(d_i-\mathfrak{p}_i\)^4\right]\\
			=&\rtr\left[\bbP^{(2)}\(\bA\circ\bB\)\bbP^{(2)}\(\bSi\circ\bSi\)\right]+\rtr\left[\(\bbP^{(2)}\circ\bSi\)\bA\(\bbP^{(2)}\circ\bSi\)\bB\right]\\
			&+\rtr\left[\(\bA\circ\bbP^{(2)}\)\bSi\(\bB\circ\bbP^{(2)}\)\bSi\right]+\rtr\left[\bA\circ\bSi\circ\bB\circ\bSi\circ\(\bbP^{(4)}-3\(\bbP^{(2)}\)^2\)\right].
		\end{align*}
		
		Thirdly, let $\bd$ denote a vector consisting of the diagonal  entries of $\bbD$ and $\boldsymbol\xi=\bd-\bp$ where $\bp $ is the expectation of $\bd$. Then we deduce
		\begin{align*}
			\mathcal{H}_3=&\rre\left[\boldsymbol\xi^T\left(\bA\circ\bSi\right)\boldsymbol\xi-\tr\left(\bbP^{(2)}(\bA\circ \bSi)\right)\right]\left[\boldsymbol\xi^T\left(\bB\circ\bSi\right)\boldsymbol\xi-\tr\left[\bbP^{(2)}(\bB\circ \bSi)\right]\right]\\
			&+2\rre\big[\rtr\left(\(\bbD-\bbP\)\bSi\(\bbD-\bbP\)\bA\right)\rtr\left(\bB\bbP\bSi\(\bbD-\bbP\)\right)\big]\\
			&+\rre\left[\rtr\left(\bA\bbP\bSi\(\bbD-\bbP\)\right)\rtr\left(\(\bbD-\bbP\)\bSi\(\bbD-\bbP\)\bB\right)\right]\\
			&+\rre\left[\rtr\left(\bA^T\bbP\bSi\(\bbD-\bbP\)\right)\rtr\left(\(\bbD-\bbP\)\bSi\(\bbD-\bbP\)\bB\right)\right]\\
			&+2\rre\left[\rtr\(\bA\bbP\bSi\(\bbD-\bbP\)\)\rtr\(\bB\bbP\bSi\(\bbD-\bbP\)\)\right]\\
			&+2\rre\left[\rtr\(\bA^T\bbP\bSi\(\bbD-\bbP\)\)\rtr\(\bB\bbP\bSi\(\bbD-\bbP\)\)\right]\\
			=&\rtr\left[\bA\circ\bSi\circ\bB\circ\bSi\circ\(\bbP^{(4)}-3\(\bbP^{(2)}\)^2\)\right]+2\rtr\left[\(\bA\circ\bSi\)\bbP^{(2)}\(\bB\circ\bSi\)\bbP^{(2)}\right]\\
			&+2\rtr\left[\bbP^{(3)}\circ\bA\circ\bSi\circ\(\bB\bbP\bSi\)\right]+\rtr\left[\bbP^{(3)}\circ\bB\circ\bSi\circ\(\bA\bbP\bSi+\bA^T\bbP\bSi\)\right]\\
			&+2\rtr\left[\bbP^{(2)}\circ\(\bA\bbP\bSi+\bA^T\bbP\bSi\)\circ\(\bB\bbP\bSi\)\right].	
		\end{align*}
		
	By combining the calculations above, we have successfully completed the proof of this lemma.
	
\subsection{Proof of Lemma \ref{lec1}}
		Using Lemma \ref{lep1}, one has
		\begin{align*}
			&\rre\left|\by^T\bbD{\bf A}\bbD\by-\rtr\({\bf A}\bT_n\)\right|^{\ell}\\
			\le& C_{\ell}\rre\left|\by^T\bbD{\bf A}\bbD\by-\rtr\(\bbD{\bf A}\bbD\bSi\)\right|^{\ell}+C_{\ell}\rre\left|\rtr\(\bbD{\bf A}\bbD\bSi\)-\rtr\({\bf A}\bT_n\)\right|^{\ell}\\
			\le& C_{\ell}\left[\left(n\nu_4\|\bA\|^2\right)^{\ell/2}+\nu_{2\ell}n\|\bA\|^{\ell}\right]+C_{\ell}\rre\left|\rtr\(\bbD{\bf A}\bbD\bSi\)-\rtr\({\bf A}\bT_n\)\right|^{\ell}.	
		\end{align*}
		Let $\bd$ denote a vector consisting of the diagonal  entries of $\bbD$ and $\boldsymbol\xi=\bd-\bp$ where $\bp $ is the expectation of $\bd$. Then we  have 
		\begin{align*}
			&\rtr\(\bbD{\bf A}\bbD\bSi\)-\rtr\({\bf A}\bT_n\)
			=\boldsymbol\xi^T\(\bA\circ\bSi\)\boldsymbol\xi+2\boldsymbol\xi^T\(\bA\circ\bSi\)\bp-\rtr\left[\bbP^{(2)}\({\bf A}\circ\bSi\)\right].
		\end{align*}
		Thus, by Lemma \ref{lep1} and Lemma \ref{lemma:8},
		\begin{align*}
			&\rre\left|\rtr\(\bbD{\bf A}\bbD\bSi\)-\rtr\({\bf A}\bT_n\)\right|^{\ell}\\	
			\le& C_{\ell}\rre\left|\boldsymbol\xi^T\(\bA\circ\bSi\)\boldsymbol\xi-\rtr\left[\bbP^{(2)}\({\bf A}\circ\bSi\)\right]\right|^{\ell}+C_{\ell}\rre\left|\boldsymbol\xi^T\(\bA\circ\bSi\)\bp\right|^{\ell}\\
			\le&C_{\ell}\left[\left(n\|\bA\|^2\right)^{\ell/2}+n\|\bA\|^{\ell}\right]+C_{\ell}\rre\left|\sum_{j=1}^p\(d_{j}-p_j\)\(\sum_{ t=1}^pp_ta_{jt}\sigma_{jt}\)\right|^{\ell}\\
			\le&C_{\ell}\left[\left(n\|\bA\|^2\right)^{\ell/2}+n\|\bA\|^{\ell}\right]+C_{\ell}\rre\left[\sum_{j=1}^p|d_{j}-p_j|^2\left|\sum_{ t=1}^pp_ta_{jt}\sigma_{jt}\right|^2\right]^{\ell/2}\\
			\le&C_{\ell}\left[\left(n\|\bA\|^2\right)^{\ell/2}+n\|\bA\|^{\ell}\right]+C_{\ell}\left[n\|\bA\|^2\right]^{\ell/2}\le C_{\ell}\left[\left(n\|\bA\|^2\right)^{\ell/2}+n\|\bA\|^{\ell}\right].
		\end{align*}
		Consequently, one concludes
		\begin{align*}
			&\rre\left|\by^T\bbD{\bf A}\bbD\by-\rtr\({\bf A}\bT_n\)\right|^{\ell}
			\le C_{\ell}\left[\left(n\nu_4\|\bA\|^2\right)^{\ell/2}+\nu_{2\ell}n\|\bA\|^{\ell}\right].	
		\end{align*}

\section{Detailed proof of Theorem \ref{thm2}}
		
	\subsection{The limiting distribution of $M_{n1}(z)$}\label{ssee}
	In this section, our objective is to determine the limiting distribution of $M_{n1}(z)$. Specifically, we aim to demonstrate that for any positive integer $r$, the sum
$$\sum_{j=1}^r\alpha_jM_{n1}(z_j) \qquad\Im z_j\neq0$$
converges in distribution to a Gaussian random variable. To achieve this, we will utilize the central limit theorem for martingale difference sequences.
	
To begin with, we introduce some notation. For $k=1,2,\cdots,n$, denote the following quatities
\begin{align*}
	&\bM_{jk}=\bM_k-\frac1n\bss\bbD_j\by_j\by_j^T\bbD_j\bss, \quad \varepsilon_k(z)= {\by_k^T\bbD_k \bss{\bM_k^{ - 1}}\bss\bbD_k\by_k}-\rtr{\bM_k^{ - 1}}\mb \Psi_n,\\&\hat\varepsilon_k(z)= {\by_k^T\bbD_k \bss{\bM_k^{ - 1}}\bss\bbD_k\by_k}-\rre\rtr{\bM_k^{ - 1}}\mb \Psi_n,\\
	&\delta_k(z)={\by_k^T\bbD_k\bss {\bM_k^{ - 2}}\bss\bbD_k\by_k}-\rtr{\bM_k^{ - 2}}\mb\Psi_n,
\end{align*} and
\begin{align*} &\tilde{\beta}_k(z)=\frac1{1+n^{-1}\rtr\left({\bM_k^{ - 1}}\mb\Psi_n\right)},\quad \beta_{jk}(z)=\frac1{1+n^{-1}\by_j^T\bbD_j\bss\bM_{jk}^{-1}\bss\bbD_j\by_j}.
\end{align*} 
Using \eqref{3ad} and \begin{align}
	\beta_k(z)=&\tilde\beta_k(z)-\frac1n\beta_k(z)	\tilde\beta_k(z)\varepsilon_k(z)\label{3ac1}\\
	=&\tilde\beta_k(z)-\frac1n	\tilde\beta_k^2(z)\varepsilon_k(z)+\frac1{n^2}\beta_k(z)	\tilde\beta_k^2(z)\varepsilon_k^2(z),\label{3ac2}
\end{align}
 we get
	\begin{align*}
		{{ M}_{n1}}\left(z\right)
		=&-\frac{1}{n}\sum\limits_{k = 1}^{n}{\rm E}_{k}\tilde\beta_k(z)\delta_k(z)+\frac{1}{n^2}\sum\limits_{k = 1}^{n}{\rm E}_{k}	\tilde\beta_k^2(z)\varepsilon_k(z)\rtr\left(\bss{\bM_k^{ - 2}}\bss\bT_n\right)+o_p(1).
	\end{align*}
	Define
	\begin{align*}
		\tau_k(z)=&\frac{1}{n}{\rm E}_{k}\tilde\beta_k(z)\delta_k(z)-\frac{1}{n^2}{\rm E}_{k}	\tilde\beta_k^2(z)\varepsilon_k(z)\rtr\left[{\bM_k^{ - 2}}\mb\Psi_n\right]
		=n^{-1}\frac{d}{dz}{\rm E}_{k}\tilde{\beta}_k(z)\varepsilon_k(z).
	\end{align*}
	Thus we only need to prove that $\sum_{j=1}^r\alpha_j\sum_{k=1}^n\tau_k(z_j)=\sum_{k=1}^n\sum_{j=1}^r\alpha_j\tau_k(z_j)$ converges in distribution to a Gaussian random variable. By Lemma \ref{lep2}, it suffices to verify conditions $(i)$ and $(ii)$. Applying  Lemma \ref{lec1}, we have
	\begin{align*}
		\sum_{k=1}^n\rre\left|\sum_{j=1}^r\alpha_j\tau_k(z_j)\right|^4\le&\frac C{n^4}\sum_{k=1}^n\sum_{j=1}^r\alpha_j^4\left[\rre|\delta_k(z_j)|^4+\rre|\varepsilon_k(z_j)|^4\right]\to0,
	\end{align*}
	which implies that condition $(ii)$ of Lemma \ref{lep2} is satisfied. Subsequently, we need to find the limit in probability of $$\Phi(z_1,z_2)\triangleq\sum_{k=1}^n\rre_{k-1}\left[\tau_k(z_1)\tau_k(z_2)\right]$$
	for $z_1,z_2$ with nonzero imaginary parts. It is obvious that
	\begin{align*}
		\Phi(z_1,z_2)=n^{-2}\frac{\partial^2}{\partial z_2\partial z_1}\sum_{k=1}^n\rre_{k-1}\left[\rre_{k}\left(\tilde{\beta}_k(z_1)\varepsilon_k(z_1)\right)\rre_{k}
		\left(\tilde{\beta}_k(z_2)\varepsilon_k(z_2)\right)\right].
	\end{align*}
	For the same reason as the analysis on page 571 in \cite{Bai2004a}, it is enough to prove that
	$$n^{-2}\sum_{k=1}^n\rre_{k-1}\left[\rre_{k}\left(\tilde{\beta}_k(z_1)\varepsilon_k(z_1)\right)\rre_{k}
	\left(\tilde{\beta}_k(z_2)\varepsilon_k(z_2)\right)\right]$$
	converges in probability to a constant and to find this limit. 
	
	Applying the martingale decomposition method as (\ref{c6}), we have for $l=1,2$
	\begin{align}\label{eq13}
		\begin{split}
			&\rre|\tilde{\beta}_k(z_l)-b_n(z_l)|^2\\
			\le&\frac C{n^2}\rre\left|\rtr\left[\bss{\bM_k^{ - 1}}\bss\bT_n\right]-\rtr\left[\bss{\bM^{ - 1}}\bss\bT_n\right]\right|^2\\&+\frac C{n^2}\rre\left|\rtr\left[\bss{\bM^{ - 1}}\bss\bT_n\right]-\rre\rtr\left[\bss{\bM^{ - 1}}\bss\bT_n\right]\right|^2
			\le\frac Cn\to0.
		\end{split}
	\end{align}
	From the above inequality, we get
	\begin{align*}
		&\rre\Bigg|n^{-2}\sum_{k=1}^n \rre_{k-1}\bigg[\rre_{k}\left(\tilde\beta_k(z_1)\varepsilon_k(z_1)\right)\rre_{k}\left(\tilde\beta_k(z_2)\varepsilon_k(z_2)\right)\\
		&-b_n(z_1)b_n(z_2)\rre_{k}\left(\varepsilon_k(z_1)\right)\rre_{k}\left(\varepsilon_k(z_2)\right)\bigg]\Bigg|\le\frac Cn\to0,
	\end{align*}
	which yields
	\begin{align*}
		&n^{-2}\sum_{k=1}^n \rre_{k-1}\bigg[\rre_{k}\left(\tilde\beta_k(z_1)\varepsilon_k(z_1)\right)\rre_{k}\left(\tilde\beta_k(z_2)\varepsilon_k(z_2)\right)\\
		&-b_n(z_1)b_n(z_2)\rre_{k}\left(\varepsilon_k(z_1)\right)\rre_{k}\left(\varepsilon_k(z_2)\right)\bigg]\xrightarrow{i.p.}0.
	\end{align*}
	Therefore, our goal turns to find the limit in probability of
	\begin{align*}
		n^{-2}b_n(z_1)b_n(z_2)\sum_{k=1}^n \rre_{k-1}\left[\rre_k\left(\varepsilon_k(z_1)\right)\rre_{k}\left(\varepsilon_k(z_2)\right)\right].
	\end{align*}
Using Lemma \ref{lec2}, we have
\begin{align*}
	&\rre_{k-1}\left[\rre_k\left(\varepsilon_k(z_1)\right)\rre_{k}\left(\varepsilon_k(z_2)\right)\right]=\mathcal{G}_{1k}(z_1,z_2)+\mathcal{G}_{2k}(z_1,z_2)+(\nu_4-3)\mathcal{G}_{3k}(z_1,z_2),
\end{align*}
where
\begin{align*}
	\mathcal{G}_{1k}(z_1,z_2)=&2\rtr\left[\rre_{k}{\bM_k^{ - 1}}(z_1)\mb\Psi_n\rre_{k}{\bM_k^{ - 1}}(z_2)\mb\Psi_n\right],
\end{align*}
\begin{align*}
	\mathcal{G}_{2k}(z_1,z_2)=\mathcal{I}_2\(\bss\rre_{k}{\bM_k^{ - 1}}(z_1)\bss,\bss\rre_{k}{\bM_k^{ - 1}}(z_2)\bss,\mb\Sigma_n,\mathbb{P}_n\).
	\end{align*}
	and 
	\begin{align*}
	\mathcal{G}_{3k}(z_1,z_2)=\mathcal{II}\(\bss\rre_{k}{\bM_k^{ - 1}}(z_1)\bss,\bss\rre_{k}{\bM_k^{ - 1}}(z_2)\bss,\mb\Sigma_n,\mathbb{P}_n\).
	\end{align*}
	Consequently, our goal is transformed into studying
	\begin{align*}
		n^{-2}b_n(z_1)b_n(z_2)\sum_{k=1}^n\Big(\mathcal{G}_{1k}(z_1,z_2)+\mathcal{G}_{2k}(z_1,z_2)+(\nu_4-3)\mathcal{G}_{3k}(z_1,z_2)\Big).
	\end{align*}

	At first, we shall derive the  limit of  $$n^{-2}b_n(z_1)b_n(z_2)\sum_{k=1}^n\mathcal{G}_{1k}(z_1,z_2).$$ 
	Let  
	${\hat\bG_n}(z)=\frac{n-1}nb_n(z)\bss\bT_n\bss-z\bI_p.$
	Then
	\begin{align}\label{cc2}
		\begin{split}
			\bM^{-1}_k(z)=&\hat\bG_n^{-1}(z)-\frac1n\sum_{j\neq k}b_n(z)\hat\bG_n^{-1}(z)\bss\(\bbD_j\by_j\by_j^T\bbD_j-\bT_n\)\bss\bM_{jk}^{-1}(z)\\
			&-\frac1n\sum_{j\neq k}\(\beta_{jk}(z)-b_n(z)\)\hat\bG_n^{-1}(z)\bss\bbD_j\by_j\by_j^T\bbD_j\bss\bM_{jk}^{-1}(z)\\
			&-\frac1nb_n(z)\sum_{j\neq k}\hat\bG_n^{-1}(z)\bss\bT_n\bss\(\bM_{jk}^{-1}(z)-\bM_k^{-1}(z)\)\\
			\triangleq&\hat\bG_n^{-1}(z)+\mathcal{A}_1(z)+\mathcal{A}_2(z)+\mathcal{A}_3(z).
		\end{split}
	\end{align}
	
	Let $\bC$ be $p\times p$ and let $\|\bC\|$ denote a nonrandom matrix bound on the spectral norm of $\bC$ for all parameters governing $\bC$ and under all realizations of $\bC$. By Lemma \ref{lec1}, one gets
	\begin{align}\label{cc3}
		\rre\left|\rtr\(\mathcal{A}_1(z)\bC\)\right|\le &\frac C{\sqrt n}\sum_{j\neq k}\rre^{1/2}\left\|\bM_{jk}^{-1}(z)\bC\hat\bG_n^{-1}(z)\right\|^2
		\le C{\sqrt n}
	\end{align}
	and
	\begin{align}\label{cc4}
		\begin{split}
			\rre\left|\rtr\(\mathcal{A}_2(z)\bC\)\right|\le &\frac Cn\sum_{j\neq k}\rre^{1/2}\left|\beta_{jk}(z)-b_n(z)\right|^2\rre^{1/2}\left|\by_j^T\by_j\right|^2\\
			\le& C\sum_{j\neq k}\rre^{1/2}\left|\beta_{jk}(z)-b_n(z)\right|^2= o(n).
		\end{split}
	\end{align}
	It is apparent that from (\ref{c4})
	\begin{align}\label{cc5}
		\rre\left|\rtr\(\mathcal{A}_3(z)\bC\)\right|\le &\frac Cn\sum_{j\neq k}\rre\left|\rtr\left(\bC\hat\bG_n^{-1}(z)\bT_n\(\bM_{jk}^{-1}(z)-\bM_k^{-1}(z)\)\right)\right|\le C.
	\end{align}
	
	Let $$\breve\bM_k(z)=\frac1n\sum_{j<k}\bss\bbD_j\by_j\by_j^T\bbD_j\bss+\frac1n\sum_{j>k}\bss\breve\bbD_j\breve\by_j\breve\by_j^T\breve\bbD_j\bss-z\bI_p$$  where $\breve\by_j,\breve\bbD_j$ are i.i.d. copies of $\by_j,\bbD_j, j = 1,\cdots,n.$
	From (\ref{cc2})-(\ref{cc5}),  one can find that
	\begin{align*}
		\begin{split}
			&\frac1n\rre_k\left[z_1\rtr\left({\bM_k^{ - 1}}(z_1)\bT_n\right)-z_2\rtr\left({\bM_k^{ - 1}}(z_2)\bT_n\right)\right]\\
			&-\frac1n\left[z_1\rtr\left({\hat\bG_n^{ - 1}}(z_1)\bT_n\right)-z_2\rtr\left({\hat\bG_n^{ - 1}}(z_2)\bT_n\right)\right]\xrightarrow{p} 0.
		\end{split}
	\end{align*}
	Furthermore, we know
	\begin{align*}
		\begin{split}
			\frac1n\left[\rtr\(z_1{\hat\bG_n^{ - 1}}(z_1)-z_2{\hat\bG_n^{ - 1}}(z_2)\)\mb\Psi_n\right]\to \frac{\underline m(z_1)-\underline m(z_2)}{\underline m(z_1)\underline m(z_2)}-\(z_1-z_2\).
		\end{split}
	\end{align*}
	Hence, it yields 
	\begin{align*}
		\frac1n\rre_k\left[\rtr\(z_1{\bM_k^{ - 1}}(z_1)-z_2{\bM_k^{ - 1}}(z_2)\)\mb\Psi_n\right]\xrightarrow{p} \frac{\underline m(z_1)-\underline m(z_2)}{\underline m(z_1)\underline m(z_2)}-\(z_1-z_2\).
	\end{align*}
	
	By calculation and the fact that
	\begin{align*}
		\beta_{jk}(z)=b_n(z)+o_{p}(1),
	\end{align*} we get
	\begin{align*}
		&\frac1n\rre_k\left[z_1\rtr\left(\bss{\bM_k^{ - 1}}(z_1)\bss\bT_n\right)-z_2\rtr\left({\bss\breve\bM_k^{ - 1}}(z_2)\bss\bT_n\right)\right]\\
		=&\frac1n\rre_k\rtr\left[\bss\bM_k^{- 1}(z_1)\left(z_1\breve\bM_k(z_2)-z_2\bM_k(z_1)\right)\breve\bM_k^{- 1}(z_2)\bss\bT_n\right]\\
		=&\frac1{n^2}\rre_k\rtr\Bigg[\bss\bM_k^{- 1}(z_1)\bss\Bigg((z_1-z_2)\sum_{j=1}^{k-1}\bbD_j\by_j\by_j^T\bbD_j+z_1\sum_{j=k+1}^{n}\breve\bbD_j\breve\by_j\breve\by_j^T\breve\bbD_j\\
		&\qquad\qquad\quad-z_2\sum_{j=k+1}^{n}\bbD_j\by_j\by_j^T\bbD_j\Bigg)\bss\breve\bM_k^{- 1}(z_2)\bss\bT_n\Bigg].
	\end{align*}
	The above quantity  can be split into \begin{align*}
 &\frac{(z_1-z_2)}{n^2}\sum_{j=1}^{k-1}\rre_k\left(\beta_{jk}(z_1)\breve\beta_{jk}(z_2)\by_j^T\bbD_j\bss\breve\bM_{jk}^{- 1}(z_2)\bss\bT_n\bss\bM_{jk}^{- 1}(z_1)\bss\bbD_j\by_j\right)\\
		&+\frac{z_1}{n^2}\sum_{j=k+1}^{n}\rre_k\(\breve\beta_{jk}(z_2)\breve\by_j^T\breve\bbD_j\bss\breve\bM_{jk}^{- 1}(z_2)\bss\bT_n\bss\bM_k^{- 1}(z_1)\bss\breve\bbD_j\breve\by_j\)\\
		&-\frac{z_2}{n^2}\sum_{j=k+1}^{n}\rre_k\(\beta_{jk}(z_1)\by_j^T\bbD_j\bss\breve\bM_{k}^{- 1}(z_2)\bss\bT_n\bss\bM_{jk}^{- 1}(z_1)\bss\bbD_j\by_j\),	
 \end{align*}
 which is proved to have a limit in probability as
 \begin{align*}
&\frac{kz_1z_2\underline m_n^0(z_1)\underline m_n^0(z_2)(z_1-z_2)}{n^2}\rre_k\rtr\left(\breve\bM_{k}^{- 1}(z_2)\mb\Psi_n\bM_{k}^{- 1}(z_1)\mb\Psi_n\right)\\
		&+\frac{z_1z_2(\underline m_n^0(z_1)-\underline m_n^0(z_2))(n-k)}{n^2}\rre_k\rtr\(\breve\bM_{k}^{- 1}(z_2)\mb\Psi_n\bM_k^{- 1}(z_1)\mb\Psi_n\).	
 \end{align*}
	This implies
	\begin{align*}
		&\frac{b(z_1)b(z_2)}{n^{2}}\sum_{k=1}^n\mathcal{G}_{1k}(z_1,z_2)	\to2\int_0^1\frac{d_1(z_1,z_2)}{1-td_1(z_1,z_2)}dt.
	\end{align*}
	
	We are now in position to  show the  limit of  $n^{-2}b_n(z_1)b_n(z_2)\sum_{k=1}^n\mathcal{G}_{2k}(z_1,z_2)$ and $n^{-2}b_n(z_1)b_n(z_2)\sum_{k=1}^n\mathcal{G}_{3k}(z_1,z_2).$
	By (\ref{cc2})-(\ref{cc5}) and
	\begin{align}\label{eqq}
		\rtr\(\bA^T\(\bB\circ\bC\)\)=\rtr\(\(\bA^T\circ\bB^T\)\bC\).
	\end{align}
We find that
	\begin{align*}
		&\frac1n \rtr\left[\bbP^{(2)}\(\bss\rre_{k}{\bM_k^{ - 1}}(z_1)\bss\circ\bSi\)\bbP^{(2)}\(\bss\rre_{k}{\bM_k^{ - 1}}(z_2)\bss\circ\bSi\)\right]\\
		=	&\frac1n \rtr\left[\bss\rre_{k}{\bM_k^{ - 1}}(z_1)\bss\(\(\bbP^{(2)}\(\bss\rre_{k}{\bM_k^{ - 1}}(z_2)\bss\circ\bSi\)\bbP^{(2)}\)\circ\bSi\)\right]\\
		=	&\frac1n \rtr\left[\bss\hat\bG_n^{-1}(z_1)\bss\(\(\bbP^{(2)}\(\bss\rre_{k}{\bM_k^{ - 1}}(z_2)\bss\circ\bSi\)\bbP^{(2)}\)\circ\bSi\)\right]\\
		&+\frac1n \rtr\left[\bss\rre_{k}\mathcal{A}_1(z_1)\bss\(\(\bbP^{(2)}\(\bss\rre_{k}{\bM_k^{ - 1}}(z_2)\bss\circ\bSi\)\bbP^{(2)}\)\circ\bSi\)\right]+o_p(1)\\
		=
		&\frac1n \rtr\left[\(\bss\hat\bG_n^{-1}(z_1)\bss\circ\bSi\)\bbP^{(2)}\(\bss\hat\bG_n^{-1}(z_2)\bss\circ\bSi\)\bbP^{(2)}\right]+o_p(1),
	\end{align*}
	where in the last equality we use the fact
	\begin{align*}
		&\frac1n \rtr\left[\bss\rre_{k}\mathcal{A}_1(z_1)\bss\(\(\bbP^{(2)}\(\bss{\bM_k^{ - 1}}(z_2)\bss\circ\bSi\)\bbP^{(2)}\)\circ\bSi\)\right]\\
		=&\frac{1}{n}\sum_{i,t}\be_i^T\bss\rre_{k}\mathcal{A}_1(z_1)\bss\be_t\be_t^T
		\bss\bM_{k}^{ - 1}(z_2)\bss\be_i\be_t^T\bbP^{(2)}\be_t\be_i^T\bbP^{(2)}\be_i\sigma_{ti}^2\\
		=&-\frac{b^2(z)}{n^3}\sum_{j< k}\sum_{i,t}\be_i^T\bss\hat\bG_n^{-1}(z_1)\bss\(\bbD_j\by_j\by_j^T\bbD_j-\bT_n\)\bss\rre_k\bM_{jk}^{-1}(z_1)\bss\be_t\\
		&\qquad\cdot\be_t^T\bss
		\bM_{jk}^{ - 1}(z_2)\mb\Theta_n\bbD_j\by_j\by_j^T\bbD_j\mb\Theta_n\bM_{jk}^{ - 1}(z_2)\bss\be_i\be_t^T\bbP^{(2)}\be_t\be_i^T\bbP^{(2)}\be_i\sigma_{ti}^2+o_p(1)\\
		=&-\frac{b^2(z)}{n^3}\sum_{j< k}\sum_{i,t}\be_i^T\bss\hat\bG_n^{-1}(z_1)\bss\bbD_j\by_j\by_j^T\bbD_j\bss\rre_k\bM_{jk}^{-1}(z_1)\bss\be_t\\
		&\qquad\cdot\be_t^T\bss
		\bM_{jk}^{ - 1}(z_2)\mb\Theta_n\bbD_j\by_j\by_j^T\bbD_j\mb\Theta_n\bM_{jk}^{ - 1}(z_2)\bss\be_i\be_t^T\bbP^{(2)}\be_t\be_i^T\bbP^{(2)}\be_i\sigma_{ti}^2+o_p(1)\\
		=&-\frac{b^2(z)}{n^3}\sum_{j< k}\sum_{i,t}\be_t^T
		\bss\bM_{jk}^{ - 1}(z_2)\bss\bbD_j\bSi\bbD_j\bss\rre_k\bM_{jk}^{-1}(z_1)\bss\be_t\\
		&\qquad\cdot\be_i^T\bss\hat\bG_n^{-1}(z_1)\bss\bbD_j\bSi\bbD_j\bss\bM_{jk}^{ - 1}(z_2)\bss\be_i\be_t^T\bbP^{(2)}\be_t\be_i^T\bbP^{(2)}\be_i\sigma_{ti}^2+o_p(1),
	\end{align*}
	which tends to zero in probability.
Taking the same procedures to analyse the other terms, it follows that
\begin{align*}
	&n^{-2}b(z_1)b(z_2)\sum_{k=1}^n\mathcal{G}_{2k}(z_1,z_2)\\
	=&n^{-2}b(z_1)b(z_2)\sum_{k=1}^n\mathcal{I}_2\(\bss\hat\bG_n^{-1}(z_1)\bss,\bss\hat\bG_n^{-1}(z_2)\bss,\bSi_n,\mathbb{P}_n\)\\
	&+n^{-2}b(z_1)b(z_2)\sum_{k=1}^n\mathcal{II}\(\bss\hat\bG_n^{-1}(z_1)\bss,\bss\hat\bG_n^{-1}(z_2)\bss,\bSi_n,\mathbb{P}_n\)+o_p(1)\\
	=&\underline m_n^0(z_1)\underline m_n^0(z_2)d_2(z_1,z_2)+o_p(1).	
\end{align*}
Therefore, we conclude 
\begin{align*}
	&\Phi(z_1,z_2)\xrightarrow{p}2\frac{\partial^2}{\partial z_2\partial z_1}\int_0^1\frac{d_1(z_1,z_2)}{1-td_1(z_1,z_2)}dt+\frac{\partial^2}{\partial z_2\partial z_1}\underline m(z_1)\underline m(z_2)d_2(z_1,z_2).
\end{align*}
Since
\begin{align*}
\frac{\partial^2}{\partial z_2\partial z_1}\int_0^1\frac{d_1(z_1,z_2)}{1-td_1(z_1,z_2)}dt=\frac{\underline m'(z_1)\underline m'(z_2)}{\(\underline m(z_2)-\underline m(z_1)\)^2}-\frac1{\(z_1-z_2\)^2},
\end{align*}
we obtain
\begin{align*}
	&\Phi(z_1,z_2)\xrightarrow{p}\frac{2\underline m'(z_1)\underline m'(z_2)}{\(\underline m(z_2)-\underline m(z_1)\)^2}-\frac2{\(z_1-z_2\)^2}+\frac{\partial^2}{\partial z_2\partial z_1}\underline m(z_1)\underline m(z_2)d_2(z_1,z_2).
\end{align*}

	\subsection{Tightness of $M_{n1}(z)$}
	
	In this section, our aim is to prove tightness of the sequence of random functions $ M_{n1}(z)$ for $z\in\mathcal{C}_n$ defined in (\ref{aaa}). Similarly to Section 3 of \cite{Bai2004a}, it suffices to show  the boundness of
	\begin{align*}
\sup_{n;z_1,z_2\in\mathcal{C}_n}\frac{\rre\left|M_{n1}(z_1)-M_{n1}(z_2)\right|^2}{|z_1-z_2|^2}.
	\end{align*}
	Using the high probability bound  of the spectral norm of $\bS_n$, it can be verified that the moments of $\|\bM^{-1}(z)\|$, $\|\bM^{-1}_j(z)\|$, and $\|\bM^{-1}_{jk}(z)\|$ are bounded in $n$ and $z\in\mathcal{C}_n$. Also, from
$	\frac1n\by_j^T\bbD_j\bss\bM^{-1}\bss\bbD_j\by_j=1-\beta_j(z),	
$
	we know
	\begin{align}\label{eqq1}
		\left|\beta_j(z)\right|
		\le1+\frac{C_1}{n}\by_j^T\by_j+{C_2}{n}{\rm P}\(\|\bS_n\|>\eta_r \ {\rm or} \ \lambda_{\min}^{\bS_n}<\eta_l\)
	\end{align}
	which implies for $\ell\ge1$
	\begin{align}\label{eqq6}
		\rre\left|\beta_j(z)\right|^{\ell}\le&C_{\ell}+C_{\ell}n^{-1}\le C.
	\end{align}
	Denote 
	$$ \rho_j(z)=\by_j^T\bbD_j\bss\bM_j^{-1}\bss\bbD_j\by_j-\rre\rtr\left(\bss\bM^{-1}\bss\bT_n\right).$$
	By Lemma \ref{lec1} and Lemma \ref{lemma:8}, we have
	\begin{align}\label{eqq2}
		\begin{split}
			\rre\left|\rho_j(z)\right|^{\ell}\le&C_{\ell}	\rre\left|\by_j^T\bbD_j\bss\bM_j^{-1}\bss\bbD_j\by_j-\rtr\left(\bss\bM_j^{-1}\bss\bT_n\right)\right|^{\ell}\\
			&+C_{\ell}	\rre\left|\rtr\left(\bss\bM_j^{-1}\bss\bT_n\right)-\rtr\left(\bss\bM^{-1}\bss\bT_n\right)\right|^{\ell}\\
			&+C_{\ell}\rre	\left|\rtr\left(\bss\bM^{-1}\bss\bT_n\right)-\rre\rtr\left(\bss\bM^{-1}\bss\bT_n\right)\right|^{\ell}\\
			\le& C_{\ell}\left[n^{\ell/2}+\nu_{2\ell}n\right].	
		\end{split}
	\end{align}
	Note that
	\begin{align}
		\beta_j(z)=&b_n(z)-\frac1n \beta_j(z)b_n(z)\rho_j(z)\label{eq21}\\
		=&b_n(z)-\frac1n b^2(z)\rho_j(z)+\frac1{n^2}\beta_j(z)b^2(z)\rho_j^2(z)\label{al:1}.
	\end{align}
	Applying (\ref{eqq6}), (\ref{eqq2}), and (\ref{eq21}), we deduce for all large $n$
	\begin{align}\label{eqq3}
		|b_n(z)|\le&|\rre\beta_j(z)|+\frac1n\left| b_n(z)\rre( \beta_j(z)\rho_j(z))\right|
		\le C_1+C_2|b_n(z)|n^{-1/2}\le \frac{C_1}{1-C_2n^{-1/2}}.
	\end{align}
	Write
	\begin{align*}
		m_n(z_1)-m_n(z_2)=\frac1p\rtr\left(\bM^{-1}(z_1)-\bM^{-1}(z_2)\right)=\frac1p\left(z_1-z_2\right)\rtr\(\bM^{-1}(z_1)\bM^{-1}(z_2)\).
	\end{align*}
	Then we have
	\begin{align*}
		\frac{M_{n1}(z_1)-M_{n1}(z_2)}{z_1-z_2}
		=&\sum_{j=1}^n\left(\rre_j-\rre_{j-1}\right)\rtr\(\bM^{-1}(z_1)\bM^{-1}(z_2)\)\\
		=&\sum_{j=1}^n\left(\rre_j-\rre_{j-1}\right)\rtr\left(\bM^{-1}(z_1)\bM^{-1}(z_2)-\bM_j^{-1}(z_1)\bM_j^{-1}(z_2)\right).
	\end{align*}
	We split the above quantity into three parts as follows \begin{align*}
	&\sum_{j=1}^n\left(\rre_j-\rre_{j-1}\right)\rtr\Bigg[\left(\bM^{-1}(z_1)-\bM_j^{-1}(z_1)\right)\left(\bM^{-1}(z_2)-\bM_j^{-1}(z_2)\right)\Bigg]\\
		&+\sum_{j=1}^n\left(\rre_j-\rre_{j-1}\right)\rtr\Bigg[\left(\bM^{-1}(z_1)-\bM_j^{-1}(z_1)\right)\bM_j^{-1}(z_2)\Bigg]\\
		&+\sum_{j=1}^n\left(\rre_j-\rre_{j-1}\right)\rtr\Bigg[\bM_j^{-1}(z_1)\left(\bM^{-1}(z_2)-\bM_j^{-1}(z_2)\right)\Bigg]\\
		=&\frac1{n^2}\sum_{j=1}^n\left(\rre_j-\rre_{j-1}\right)\beta_j(z_1)\beta_j(z_2)\left(\by_j^T\bbD_j\bss\bM_j^{-1}(z_1)\bM_j^{-1}(z_2)\bss\bbD_j\by_j\right)^2\\
		&-\frac1{n}\sum_{j=1}^n\left(\rre_j-\rre_{j-1}\right)\beta_j(z_1)\by_j^T\bbD_j\bss\bM_j^{-2}(z_1)\bM_j^{-1}(z_2)\bss\bbD_j\by_j\\
		&-\frac1{n}\sum_{j=1}^n\left(\rre_j-\rre_{j-1}\right)\beta_j(z_2)\by_j^T
		\bbD_j\bss\bM_j^{-1}(z_1)\bM_j^{-2}(z_2)\bss\bbD_j\by_j\\
		\triangleq&\mathcal{P}_1+\mathcal{P}_2+\mathcal{P}_3.
	\end{align*}
It suffices to show that $\rre\left|\mathcal{P}_1+\mathcal{P}_2+\mathcal{P}_3\right|^2$ is bounded. In the following, we only prove that $\rre|\mathcal{P}_{1}|^2\le C$. The others are similar thus omitted.
	
	Using (\ref{eq21}), write
	\begin{align*}
		&\mathcal{P}_1\\
		=&\frac1{n^2}\sum_{j=1}^n b(z_1)b(z_2)\left(\rre_j-\rre_{j-1}\right)\left(\by_j^T\bbD_j\bss\bM_j^{-1}(z_1)\bM_j^{-1}(z_2)\bss\bbD_j\by_j\right)^2\\
		&-\frac1{n^3}\sum_{j=1}^n b(z_1)b(z_2)\left(\rre_j-\rre_{j-1}\right) \beta_j(z_2)\rho_j(z_2)\left(\by_j^T\bbD_j\bss\bM_j^{-1}(z_1)\bss\bM_j^{-1}(z_2)\bbD_j\by_j\right)^2\\
		&-\frac1{n^3}\sum_{j=1}^n b(z_1)\left(\rre_j-\rre_{j-1}\right) \beta_j(z_1)\beta_j(z_2)\rho_j(z_1)\left(\by_j^T\bbD_j\bss\bM_j^{-1}(z_1)\bM_j^{-1}(z_2)\bss\bbD_j\by_j\right)^2\\
		\triangleq&\mathcal{P}_{11}+\mathcal{P}_{12}+\mathcal{P}_{13}.
	\end{align*}
	Denote
$$\Delta_j(z_1,z_2)=\by_j^T\bbD_j\bss\bM_j^{-1}(z_1)\bM_j^{-1}(z_2)\bss\bbD_j\by_j-\rtr\(\bss\bM_j^{-1}(z_1)\bM_j^{-1}(z_2)\bss\bT_n\).	
$$
	By Lemma \ref{lec1}, we deduce that
	\begin{align*}
		\rre|\mathcal{P}_{11}|^2
		\le&\frac C{n^4}\sum_{j=1}^n \rre\Bigg|\left(\by_j^T\bbD_j\bss\bM_j^{-1}(z_1)\bM_j^{-1}(z_2)\bss\bbD_j\by_j\right)^2\\
		-&\left[\rtr\(\bss\bM_j^{-1}(z_1)\bM_j^{-1}(z_2)\bss\bT_n\)\right]^2\Bigg|^2\\
		\le&\frac C{n^4}\sum_{j=1}^n \rre\left|\Delta_j(z_1,z_2)\right|^4+\frac C{n^2}\sum_{j=1}^n \rre\left|\Delta_j(z_1,z_2)\right|^2
		\le C.
	\end{align*}
	Using  (\ref{eqq1}), (\ref{eqq2}), and Lemma \ref{lec1}, one finds
	\begin{align*}
		\rre|\mathcal{P}_{12}|^2\le
		&\frac C{n^6}\sum_{j=1}^n\rre\left| \beta_j(z_2)\rho_j(z_2)\left(\by_j^T\bbD_j\bss\bM_j^{-1}(z_1)\bM_j^{-1}(z_2)\bss\bbD_j\by_j\right)^2\right|^2\\
		\le&\frac C{n^6}\sum_{j=1}^n\rre\left| \rho_j(z_2)\Delta_j^2(z_1,z_2)\right|^2+\frac C{n^2}\sum_{j=1}^n\rre\left|\rho_j(z_2)\right|^2\\
		&+\frac C{n^8}\sum_{j=1}^n\rre\left|\(\by_j^T\by_j-\rtr\(\bSi\)\) \rho_j(z_2)\Delta_j^2(z_1,z_2)\right|^2\\
		&+\frac C{n^4}\sum_{j=1}^n\rre\left|\(\by_j^T\by_j-\rtr\(\bSi\)\)\rho_j(z_2)\right|^2+ C{n^9}{\rm P}\(\|\bS_n\|>\eta_r \ {\rm or} \ \lambda_{\min}^{\bS_n^1}<\eta_l\)
		\le C.
	\end{align*}
	$\mathcal{P}_{13}$ can be handled in the same way as $\mathcal{P}_{12}$. Hence, we obtain
	\begin{align*}
		\rre|\mathcal{P}_{1}|^2\le C.
	\end{align*}
	Finally, we conclude that
	\begin{align*}
		\sup_{n;z_1,z_2\in\mathcal{C}_n}\frac{\rre\left|M_{n1}(z_1)-M_{n1}(z_2)\right|^2}{|z_1-z_2|^2}\le\sup_{n;z_1,z_2\in\mathcal{C}_n}
		\rre\left|\mathcal{P}_1+\mathcal{P}_2+\mathcal{P}_3\right|^2\le C.
	\end{align*}
	This completes the proof of tightness.
	
	\subsection{Convergence of $M_{n2}(z)$}
	Firstly, it can be verified that $\bG_n^{-1}(z)$ is uniformly bounded on $\mathcal{C}_n$, i.e.,
	\begin{align}\label{al12}
		\sup_{n,\mathcal{C}_n}\left\|\bG_n^{-1}(z)\right\|<\infty .
	\end{align}
Then, note that
	\begin{align*}
		M_{n2}(z)=&p\rre m_n(z)-\rtr\(\bG_n^{-1}(z)\)+\rtr\(\bG_n^{-1}(z)\)-pm_n^0(z)\\
		=&p\rre m_n(z)-\rtr\(\bG_n^{-1}(z)\)+p\(b_n(z)+z{\underline m}_n^0(z)\)\int\frac{t}{\(b_n(z)t-z\)\(z{\underline m}_n^0(z)t+z\)}dH_n(t).
	\end{align*}
	By (\ref{ll1}) and (\ref{eq21}), it follows that
	\begin{align*}
		p\(b_n(z)+z{\underline m}_n^0(z)\)=&p\(-z\rre\underline m_n(z)+z{\underline m}_n^0(z)\)+\frac p{n^2}\sum_{k=1}^nb_n(z)\rre\beta_k(z)\rho_k(z)\\
		=&p\(-z\rre\underline m_n(z)+z{\underline m}_n^0(z)\)+o(1).
	\end{align*}
	Hence,  together with (\ref{ll2}), we obtain
	\begin{align}\label{imp}
		\begin{split}
			M_{n2}(z)
			=&-z\underline m_n^0(z)\left[1-y_n\int\frac{\(\underline m_n^0(z)\)^2t^2}{\({\underline m}_n^0(z)t+1\)^2}dH_n(t)\right]^{-1}\\
			&\qquad\cdot\(p\rre m_n(z)-\rtr\(\bG_n^{-1}(z)\)\)+o(1).
		\end{split}
	\end{align}	
	Recalling (\ref{c7}), we have
	\begin{align*}
		&p\rre m_n(z)-\rtr\(\bG_n^{-1}(z)\)\\
		=&-\frac1{n}\sum_{k=1}^n\rre\left[\beta_k(z)\left(\by_k^T\bbD_k\bss\bM_k^{-1}\bG_n^{-1}\bss\bbD_k\by_k-\rtr\bss\bM_k^{-1}\bG_n^{-1}\bss\bT_n\right)\right]\\
		&-\frac1{n}\sum_{k=1}^n\rre\left[\beta_k(z)\left(\rtr\bss\bM_k^{-1}\bG_n^{-1}\bss\bT_n-\rre\rtr\bss\bM_k^{-1}\bG_n^{-1}\bss\bT_n\right)\right]\\
		&-\frac1{n}\sum_{k=1}^n\bigg(\rre\beta_k(z)\bigg)\left[\rre\rtr\bss\bM_k^{-1}\bG_n^{-1}\bss\bT_n
		-\rre\rtr\bss\bM^{-1}\bG_n^{-1}\bss\bT_n\right]\notag\\
		&-\frac1{n}\sum_{k=1}^n\bigg(\rre\beta_k(z)-b_n(z)\bigg)\rre\rtr\bss\bM^{-1}\bG_n^{-1}\bss\bT_n\notag\\
		\triangleq&\mathcal{J}_1+\mathcal{J}_2+\mathcal{J}_3+\mathcal{J}_4\notag.
	\end{align*}
	From (\ref{al:1}),
	$\mathcal{J}_1$ equals to the sum of $$\mathcal{J}_{11}=\frac{b^2(z)}{n^2}\sum_{k=1}^n \rre\(\rho_k(z)\left(\by_k^T\bbD_k\bss\bM_k^{-1}\bG_n^{-1}\bss\bbD_k\by_k-\rtr\bM_k^{-1}\bG_n^{-1}\mb\Psi_n\right)\)$$ and $$\mathcal{J}_{12}=\frac{b^2(z)}{n^3}\sum_{k=1}^n\rre\bigg[\beta_k (z)\rho_k^2(z) \left(\by_k^T\bbD_k\bss\bM_k^{-1}\bG_n^{-1}\bss\bbD_k\by_k-\rtr\bM_k^{-1}\bG_n^{-1}\mb\Psi_n\right)\bigg].$$ 
	Applying Lemma \ref{lec1}, (\ref{eqq1}), and (\ref{al12}), we get
	\begin{align*}
		&\left|\mathcal{J}_{12}\right|\\ \le&\frac C{n^3}\sum_{k=1}^n\rre^{1/2}\left|\rho_k(z) \right|^4\rre^{1/2} \left|\by_k^T\bbD_k\bss\bM_k^{-1}\bG_n^{-1}\bss\bbD_k\by_k-\rtr\left(\bM_k^{-1}\bG_n^{-1}\mb\Psi_n\right)\right|^2\\
		&+\frac C{n^4}\sum_{k=1}^n\rre^{1/4}\left|\by_k^T\by_k-\rtr\(\bSi\)\right|^4\rre^{1/4}\left|\rho_k(z) \right|^8\\
		&\qquad\times
		\left(\rre \left|\by_k^T\bbD_k\bss\bM_k^{-1}\bG_n^{-1}\bss\bbD_k\by_k-\rtr\left(\bM_k^{-1}\bG_n^{-1}\mb\Psi_n\right)\right|^2\right)^{1/2}\\
		&+Cn^5{\rm P}\(\|\bS_n\|>\eta_r \ {\rm or} \ \lambda_{\min}^{\bS_n^1}<\eta_l\)
		\le \frac C{n}\to0.
	\end{align*}
	Moreover,  one finds
	\begin{align*}
		\mathcal{J}_{11}
		=&\frac{b^2(z)}{n^2}\sum_{k=1}^n\rre\bigg[\(\by_k^T\bbD_k\bss\bM_k^{-1}\bss\bbD_k\by_k-\rtr\left(\bM_k^{-1}\mb\Psi_n\right)\)\\
		&\qquad\cdot\left(\by_k^T\bbD_k\bss\bM_k^{-1}\bG_n^{-1}\bss\bbD_k\by_k-\rtr\left(\bM_k^{-1}\bG_n^{-1}\mb\Psi_n\right)\right)\bigg]+o(1).
	\end{align*}
	Hence, by Lemma \ref{lec2}, we see
	\begin{align*}
		\mathcal{J}_{1}
		=&\frac{b^2(z)}{n}\Bigg[\mathcal{I}_1\(\bss\bM^{-1}\bss,\bss\bM^{-1}\bG_n^{-1}\bss,\bSi_n,\mathbb{P}_n\)\\
		&+\mathcal{I}_2\(\bss\bM^{-1}\bss,\bss\bM^{-1}\bG_n^{-1}\bss,\bSi_n,\mathbb{P}_n\)\\&+\mathcal{II}\(\bss\bM^{-1}\bss,\bss\bM^{-1}\bG_n^{-1}\bss,\bSi_n,\mathbb{P}_n\)\Bigg]+o(1).	
\end{align*}
	Applying (\ref{eq21}), it follows that
	\begin{align*}
		\mathcal{J}_2	=
		&\frac{b_n(z)}{n^2}\sum_{k=1}^n\rre\left[\beta_k(z)\rho_k(z)\left(\rtr\left(\bM_k^{-1}\bG_n^{-1}\mb\Psi_n\right)-\rre\rtr\left(\bM_k^{-1}\bG_n^{-1}\mb\Psi_n\right)\right)\right].
	\end{align*}
	By martingale difference method, (\ref{eqq6})-(\ref{eqq2}),  and Lemma \ref{lec1}, we have
	\begin{align*}
		\left|\mathcal{J}_{2}\right|\le	&\frac C{n^2}\sum_{k=1}^n\rre^{1/2}\left|\rtr\left(\bM_k^{-1}\bG_n^{-1}\mb\Psi_n\right)-\rre\rtr\left(\bM_k^{-1}\bG_n^{-1}\mb\Psi_n\right)\right|^2\rre^{1/4}\left|\beta_k(z)\right|^4\rre^{1/4}\left|\rho_k(z)\right|^4
		\le\frac C{ n}.
	\end{align*}
	This yields
$
		\mathcal{J}_2	=o(1).
$
	By (\ref{eq21}), we deduce
	\begin{align*}
		\mathcal{J}_3=
		&-\frac{b_n(z)}{n}\sum_{k=1}^n\left[\rre\rtr\left(\bM_k^{-1}\bG_n^{-1}\mb\Psi_n\right)
		-\rre\rtr\left(\bM^{-1}\bG_n^{-1}\mb\Psi_n\right)\right]+o(1)\\	
		=&\frac{b^2(z)}{n^2}\sum_{k=1}^n\rre\left[\rtr\left(\bM_k^{-1}\bG_n^{-1}\mb\Psi_n\bM_k^{-1}\mb\Psi_n\right)\right]+o(1)\\
		=&\frac{b^2(z)}{n}\rre\left[\rtr\left(\bM^{-1}\bG_n^{-1}\mb\Psi_n\bM^{-1}\mb\Psi_n\right)\right]+o(1),
	\end{align*}
	and
$
		\mathcal{J}_4=
		\frac{b_n(z)}{n^2}\sum_{k=1}^n\rre\bigg(\beta_k(z)\rho_k(z)\bigg)\rre\rtr\left(\bM^{-1}\bG_n^{-1}\mb\Psi_n\right)=o(1).
$
	Therefore, we conclude 
	\begin{align}\label{ll4}
		\begin{split}
			&p\rre m_n(z)-\rtr\(\bG_n^{-1}(z)\)
				=
	\frac{b^2(z)}{n}\rre\rtr\(\bM^{-1}\mb\Psi_n\bM^{-1}\bG_n^{-1}\mb\Psi_n\)
	-\frac{b^2(z)}{z^3}a_n(z)+o(1),
	\end{split}
\end{align}
where
	\begin{align*}
		a_n(z)=-\frac{z^3}{n}\Bigg[&\mathcal{I}_1\(\bss\bM^{-1}\bss,\bss\bM^{-1}\bG_n^{-1}\bss,\bSi_n,\mathbb{P}_n\)\\
		&+\mathcal{I}_2\(\bss\bM^{-1}\bss,\bss\bM^{-1}\bG_n^{-1}\bss,\bSi_n,\mathbb{P}_n\)\\&+(\nu_4-3)\mathcal{II}\(\bss\bM^{-1}\bss,\bss\bM^{-1}\bG_n^{-1}\bss,\bSi_n,\mathbb{P}_n\)\Bigg].	
\end{align*}	
  	Note that
	\begin{align*}
		\begin{split}
			\bM^{-1}(z)=&\bG_n^{-1}(z)-\frac{b_n(z)}n\sum_{ k=1}^n\bG_n^{-1}(z)\bss\(\bbD_k\by_k\by_k^T\bbD_k-\bT_n\)\bss\bM_{k}^{-1}(z)\\
			&-\frac{b_n(z)}{n^2}\sum_{ k=1}^n\beta_{k}(z)\rho_k(z)\bG_n^{-1}(z)\bss\bbD_k\by_k\by_k^T\bbD_k\bss\bM_{k}^{-1}(z)\\
			&-\frac{b_n(z)}{n^2}\sum_{ k=1}^n\beta_k(z)\bG_n^{-1}(z)\bT_n\bM_{k}^{-1}(z)\bss\bbD_k\by_k\by_k^T\bbD_k\bss\bM_k^{-1}(z)\\
			\triangleq&\bG_n^{-1}(z)+\mathcal{B}_1(z)+\mathcal{B}_2(z)+\mathcal{B}_3(z).
		\end{split}
	\end{align*}
	Let ${\bC}, {\bC_1}$ be $p\times p$ matrices, nonrandom and bounded on the spectral norm. Applying Lemma \ref{lec1}, (\ref{eqq6}), and (\ref{eqq2}), we obtain
	\begin{align*}
		\begin{split}
			\rre\bigg|\rtr\big({\mathcal{B}_2}(z)&{\bC}\big)\bigg|
			\le\frac C{n^2}\sum_{k=1}^n\rre^{1/2}\left|\rho_k(z)\right|^2\rre^{1/4}\left|\beta_k(z)\right|^4\\
			&\cdot\rre^{1/4}\left|\by_k^T\bbD_k\bss\bM_{k}^{-1}(z)\bG_n^{-1}(z)\bss\bbD_k\by_k\right|^4
			=O(n^{1/2})
		\end{split}
	\end{align*}
	and
	\begin{align*}
		\begin{split}
			&\rre\left|\rtr\left({\mathcal{B}_3}(z){\bC}\right)\right|\\\le&\frac C{n^2}\sum_{k=1}^n\rre^{1/2}\left|\by_k^T\bbD_k\bss\bM_k^{-1}(z)\bG_n^{-1}(z)\bT_n\bM_{k}^{-1}(z)\bss\bbD_k\by_k\right|^2\rre^{1/2}\left|\beta_k(z)\right|^2
			\le C.
		\end{split}
	\end{align*}
Furthermore, we get
	\begin{align*}
		&\rre\rtr\({\mathcal{B}}_1(z)\mb\Psi_n \bM^{-1}\bC_1\)\\
		=&-\frac{b_n(z)}n\sum_{ k=1}^n\rre\rtr\(\bss\(\bbD_k\by_k\by_k^T\bbD_k-\bT_n\)\bss\bM_{k}^{-1}\mb\Psi_n \bM_k^{-1}\bC_1\bG_n^{-1}\)\\
		&-\frac{b_n(z)}n\sum_{ k=1}^n\rre\rtr\(\bG_n^{-1}\bss\bbD_k\by_k\by_k^T\bbD_k\bss\bM_{k}^{-1}\mb\Psi_n \(\bM^{-1}-\bM_k^{-1}\)\bC_1\)\\
		&+\frac{b_n(z)}n\sum_{ k=1}^n\rre\rtr\(\bG_n^{-1}(z)\bss\bT_n\bss\bM_{k}^{-1}\mb\Psi_n \(\bM^{-1}-\bM_k^{-1}\)\bC_1\)
		\\
		=
		&\frac{b^2(z)}{n}\rre\rtr\(\bM^{-1}\mb\Psi_n\bM^{-1}\mb\Psi_n\)\rre\rtr\(\bM^{-1}\bC_1\bG_n^{-1}\mb\Psi_n\)+O(\sqrt n).
	\end{align*}
	Consequently,  we see
	\begin{align*}
		\begin{split}
			&\frac1n\rre\rtr\( \bM^{-1}\mb\Psi_n\bM^{-1}\mb\Psi_n\)\\
			=&\frac1{n}\rre\rtr\( \bG_n^{-1}\mb\Psi_n \bM^{-1}\mb\Psi_n\)+\frac1{n}\rre\rtr\( \mathcal{B}_1(z)\mb\Psi_n \bM^{-1}\mb\Psi_n\)+o(1)\\
			=&\frac1{n}\rtr\( \bG_n^{-1}\mb\Psi_n\)^2+\frac{b^2(z)}{n^2}\rre\rtr\(\bM^{-1}\mb\Psi_n \bM^{-1}\mb\Psi_n\)\rtr\(\bG_n^{-1}\mb\Psi_n\)^2+o(1)\\
			=&\(1-\frac{b^2(z)}{n}\rtr\(\bG_n^{-1}\mb\Psi_n\)^2\)^{-1}\frac1{n}\rtr\( \bG_n^{-1}\mb\Psi_n \)^2+o(1)
		\end{split}
	\end{align*}
	and
	\begin{align}\label{ll5}
		\begin{split}
			&\frac{b^2(z)}{n}\rre\rtr\( \bM^{-1}\mb\Psi_n \bM^{-1}\bG_n^{-1}\mb\Psi_n\)\\
			=&\frac{b^2(z)}{n}\rre\rtr\( \bG_n^{-1}\mb\Psi_n \bM^{-1}\bG_n^{-1}\mb\Psi_n\)+\frac{b^2(z)}{n}\rre\rtr\( \mathcal{B}_1(z)\mb\Psi_n \bM^{-1}\bG_n^{-1}\mb\Psi_n\)+o(1)\\
			=&\frac{b^2(z)}{n}\rtr\( \bG_n^{-1}\mb\Psi_n \bG_n^{-2}\mb\Psi_n\)\(1+\frac{b^2(z)}{n}\rre\rtr\(\bM^{-1}\mb\Psi_n \bM^{-1}\mb\Psi_n\)\)+o(1)\\
			=&\frac{b^2(z)}{n}\rtr\( \bG_n^{-1}\mb\Psi_n \bG_n^{-2}\mb\Psi_n\)\(1-\frac{b^2(z)}{n}\rtr\(\bG_n^{-1}\mb\Psi_n\)^2\)^{-1}+o(1).
		\end{split}
	\end{align}
	 It is also easy to see that

	\begin{align*}
		a_n(z)=-\frac{z^3}{n}\Bigg[&\mathcal{I}_1\(\bss\bG_n^{-1}\bss,\bss\bG_n^{-2}\bss,\bSi_n,\mathbb{P}_n\)\\
		&+\mathcal{I}_2\(\bss\bG_n^{-1}\bss,\bss\bG_n^{-2}\bss,\bSi_n,\mathbb{P}_n\)\\&+(\nu_4-3)\mathcal{II}\(\bss\bG_n^{-1}\bss,\bss\bG_n^{-2}\bss,\bSi_n,\mathbb{P}_n\)\Bigg]+o(1).	
	\end{align*}
	Combining  (\ref{ll4})-(\ref{ll5}) and the above equality, we find
	\begin{align*}
		\begin{split}
			&p\rre m_n(z)-\rtr\(\bG_n^{-1}(z)\)\\
			=&\frac{b^2(z)}{n}\rtr\( \bG_n^{-1}\mb\Psi_n \bG_n^{-2}\mb\Psi_n\)\(1-\frac{b^2(z)}{n}\rtr\(\bG_n^{-1}\mb\Psi_n\)^2\)^{-1}\\
			&+\frac{b^2(z)}{n}\mathcal{I}_2\(\bss\bG_n^{-1}\bss,\bss\bG_n^{-2}\bss,\bSi_n,\mathbb{P}_n\)\\&+\frac{(\nu_4-3)b^2(z)}{n}\mathcal{II}\(\bss\bG_n^{-1}\bss,\bss\bG_n^{-2}\bss,\bSi_n,\mathbb{P}_n\)
			+o(1).
		\end{split}
	\end{align*}
	From (\ref{imp}) and the above equality, it yields
	\begin{align*}
		M_{n2}(z)
		\to&\frac{y\int\frac{\underline m^3(z)t^2}{\({\underline m}(z)t+1\)^3}dH(t)}{\(1-y\int\frac{\underline m^2(z)t^2}{\({\underline m}(z)t+1\)^2}dH(t)\)^2}+\frac{\underline m^3(z)a(z)}{1-y\int\frac{\underline m^2(z)t^2}{\({\underline m}(z)t+1\)^2}dH(t)}.
	\end{align*}
	Utilizing the arguments presented on pages 592-593 of \cite{Bai2004a}, one could  establish the uniform equicontinuity of $M_{2n}(z)$ on $\mathcal{C}_n$. Omitting this analogous step, we conclude the proof of uniform convergence.
	
The proof of Theorem \ref{thm2} is thus complete.


\section{Lemmas}
The necessary lemmas are listed in this section.

	\begin{lem}[Theorem A.44 of Bai and Silverstein (2010)]\label{lemma:3}
		Let ${\mathbf A}$ and ${\mathbf B}$ be $p \times n$ complex matrices. Then, $$\left\| {{F^{{\mathbf A}{{\mathbf A}^ * }}} - {F^{{\mathbf B}{{\mathbf B}^ * }}}} \right\|_{KS} \le \frac{1}{p}{\rm rank}({\mathbf A} - {\mathbf B}).$$
	\end{lem}
	
	\begin{lem}[Lemma 2.4 of  \cite{Silverstein95S}]\label{le9}
		For $n\times n$ Hermitian $\bA$ and $\bB$,
		\begin{align*}
			\|F^{\bA}-F^{\bB}\|_{KS}\le\frac1n{\rm rank}(\bA-\bB).
		\end{align*}
	\end{lem}

	\begin{lem}[Corollary A.42 of \cite{BaiS10S}] \label{lemma:4}
		Let ${\mathbf A}$ and ${\mathbf B}$ be two $p \times n$  matrices and denote the ESD of ${\mathbf S} = {\mathbf A}{{\mathbf A}^ * }$ and $ \widetilde{\mathbf S} = {\mathbf B}{{\mathbf B}^ * }$   by ${F^{\mathbf S}}$ and ${F^{\widetilde{\mathbf S}}}$, respectively.  Then,
		$${L^4}\left({F^{\mathbf S}},{F^{\widetilde{\mathbf S}}}\right) \le \frac{2}{{{p^2}}}\left({\rm tr}\left({\mathbf A}{{\mathbf A}^ * } + {\mathbf B}{{\mathbf B}^ * }\right)\right)\left({\rm tr}\left[({\mathbf A} - {\mathbf B}\right){\left({\mathbf A} - {\mathbf B}\right)^ * }\right]),$$
		where $L\left(\cdot,\cdot \right)$  denotes the L\'{e}vy distance, that is,
		$$L\left({F^{\mathbf S}},{F^{\widetilde{\mathbf S}}}\right)=\inf\{\varepsilon:{F^{\mathbf S}\left(x-\varepsilon,y-\varepsilon\right)-\varepsilon}\le{F^{\widetilde{\mathbf S}}}\le{F^{\mathbf S}\left(x+\varepsilon,y+\varepsilon\right)+\varepsilon}\}.$$
	\end{lem}
	\begin{lem}[Bernstein's inequality]\label{lemma:10}
		If \ $\boldsymbol{\tau}_1,\cdots,\boldsymbol{\tau}_n$ are independent random variables with means zero and uniformly bounded by $b$, then, for any $\varepsilon > 0$,
		\begin{displaymath}
			{\rm P}\left(\left|\sum_{j=1}^{n}\boldsymbol{\tau}_j\right|\ge \varepsilon\right)\le 2{\rm exp}\left(-\varepsilon^2/\left[2\left(B_n^2+b\varepsilon\right)\right]\right),
		\end{displaymath}
		where $ B_n^2={\rm E}\left(\boldsymbol{\tau}_1+\cdots+\boldsymbol{\tau}_n\right)^2.$
	\end{lem}
	\begin{lem}[Rosenthal's inequality ]\label{lemma:5}
		Let $ {\mathbf  X}_i $ be independent with zero means,  then we have, for some constant $C_k$,
		$${\rm E}\left|\sum_i {\mathbf  X}_i\right|^{2k} \leq C_k\left(\sum_i{\rm E}\left|{\mathbf  X}_i\right|^{2k}+\left(\sum_i {\rm E}\left|{\mathbf  X}_i\right|^2\right)^k\right). $$
	\end{lem}
	
	\begin{lem}[Burkholder's inequality ]\label{lemma:8}
		Let $\{ {{\mathbf X}_k}\} $ be a complex martingale difference sequence with respect to the increasing $\sigma$-field. Then, for $p > 1,$ $${\rm E}{\left| {\sum_k {{{\mathbf  X}_k}} } \right|^p} \le {K_p}{\rm E}{\left({\sum_k{\left| {{{\mathbf  X}_k}} \right|} ^2}\right)^{p/2}}.$$
	\end{lem}
	
	\begin{lem}\label{le4}
		For rectangular matrices $\bA$, complex vector $\bf{a}$,
		\begin{align*}
			\left|\bf{a}^T\bA\bf{a}\right|\le\|\bA\|\bf{a}^T\bf{a}.
		\end{align*}
	\end{lem}
	
	\begin{lem}\label{le5}
		For Hermitian matrices $\bA,\bB$,
		\begin{align*}
			\left|\rtr\(\bA\bB\)^2\right|\le\|\bA\|^2\rtr\(\bB^2\).
		\end{align*}
	\end{lem}
	
	\begin{lem}[Lemma B.26 in \cite{BaiS10S}]\label{lep1}
		Let ${\bf A}=(a_{jk})$ be an $n\times n$ nonrandom matrix and $\bx=(x_1,\cdots,x_n)'$ be a random vector of independent entries. Assume that $\rre x_j=0$, $\rre|x_j|^2=1$ and $\rre|x_j|^l\le \nu_l$. Then for $p\ge1$,
		\begin{align*}
			&\rre\left|\bx^T{\bf A}\bx-\rtr{\bf A}\right|^p
			\le C_p\left[\left(\nu_4\rtr{\bf A}{\bf A}^T\right)^{p/2}+\nu_{2p}\rtr\left({\bf A}{\bf A}^T\right)^{p/2}\right],
		\end{align*}
		where $C_p$ is a constant depending on $p$ only.
	\end{lem}
\begin{lem}[Lemma 2.6 of \cite{Silverstein95S}]\label{le2}
	For $z=u+iv\in\mathbb{C}^+$, let $m_1(z),m_2(z)$ be the Stieltjes transforms of any two p.d.f.'s, $\bA$ and $\bB$ $n\times n$ with $\bA$ Hermitian nonnegative definite and $\br\in\mathbb{C}^n$. Then
	\begin{align*}
		(a)&\quad \|(m_1(z)\bA+\bI_n)^{-1}\|\le\max(4\|\bA\|/v,2),\\
		(b)&\quad |\br^*\bB(m_1(z)\bA+\bI_n)^{-1}\br-\br^*\bB(m_1(z)\bA+\bI_n)^{-1}\br|\\
		&\quad \quad \quad \le|m_2(z)-m_1(z)| \ \|r\|_2^2 \ \|\bB\| \ \|\bA\| \ \max(4\|\bA\|/v,2)^2,\notag
	\end{align*}
	where $\|\br\|_2$ denotes the Euclidean norm on $\br$.
\end{lem}
\begin{lem}[Kolmogorov's inequality for submartingales]\label{le7}
	If $\{X_j,\mathcal{F}_j,1\le j\le n\}$ is a martingale, then for each $l\ge1$ and $\alpha>0$,
	\begin{align*}
		\alpha^l{\rm P}\left(\max_{j\le n}|X_j|>\alpha\right)\le\rre|X_n|^l.
	\end{align*}
\end{lem}

\begin{lem}[Lemma 6.8 in \cite{BaiS10S}]\label{le8}
	If, for all $t>0$, $t^l{\rm P}(|X|>t)\le C$ for some positive $l$, then for any positive $q<l$,
	\begin{align*}
		\rre|X|^q\le C^{q/l}\left(\frac l{l-q}\right).
	\end{align*}
\end{lem}

	\begin{lem}[CLT for Martingale]\label{lep2}
		Suppose for each $n$, $\{Y_{n1},Y_{n2},\cdots,Y_{nr_n}\}$ is a real martingale difference sequence with respect to the increasing $\sigma$-field $\{\mathcal{F}_{nj}\}$ having second moments. If as $n\to\infty$,
		\begin{align*}
			&(i)\quad\quad\quad\quad\qquad\qquad\sum_{j=1}^{r_n}\rre(Y_{nj}^2|\mathcal{F}_{n,j-1})\xrightarrow{i.p.}\sigma^2,\qquad\qquad\qquad\qquad\qquad\qquad\\
			&{\rm where}\ \sigma^2\ {\rm is\ a\ positive\ constant}, \ {\rm and\ for\ each}\ \varepsilon\ge0,\\
			&(ii)\quad\quad\quad\quad\qquad\qquad\sum_{j=1}^{r_n}\rre(Y_{nj}^2I(|Y_{nj}\ge\varepsilon|))\to0,\qquad\qquad\qquad\qquad\qquad\qquad\\
			&{\rm then}\\
			&\quad\quad\quad\quad\qquad\qquad\qquad\sum_{j=1}^{r_n}Y_{nj}\xrightarrow{D}N(0,\sigma^2).\qquad\qquad\qquad\qquad
		\end{align*}
	\end{lem}